\pdfoutput=1

\documentclass[final, reqno, 10pt]{amsart}

 \usepackage{amsmath,amsthm,amsfonts}
\usepackage{latexsym,mathabx}
\usepackage{amssymb}
\usepackage{bbold}
\usepackage{enumerate}

\newtheorem{theorem}{Theorem}[section]
\newtheorem{lemma}[theorem]{Lemma}
\newtheorem{prop}[theorem]{Proposition}
\newtheorem{corollary}[theorem]{Corollary}
\newtheorem{cor}[theorem]{Corollary}
\theoremstyle{remark}

\newtheorem{definition}{Definition}[section]
\newcommand{\set}{\mathbb}
\newcommand{\dl}{\nabla}
\newcommand{\les}{\lesssim}
\renewcommand{\frak}{\mathfrak}

\newcommand{\mc}{\mathcal}
\newcommand{\be}{\begin{equation}}
\newcommand{\ee}{\end{equation}}
\newcommand{\bee}{\begin{align}}
\newcommand{\eee}{\end{align}}
\newcommand{\ba}{\begin{array}}

\newcommand{\ea}{\end{array}}
\newcommand{\bpm}{\begin{pmatrix}}
\newcommand{\epm}{\end{pmatrix}}
\newcommand{\lb}{\label}
\newcommand{\la}{\langle}
\newcommand{\ra}{\rangle}
\def\calH{\mc{H}}
\def\calL{\mc{L}}
\def\wtil{\widetilde}

\DeclareMathOperator{\Ran}{Ran}

\DeclareMathOperator{\Imim}{Im}

\newcommand{\oast}{\circledast}
\DeclareMathOperator*{\slim}{s-lim}

\newcommand{\ov}{\overline}
\newcommand{\dd}{{\,}{d}}

\renewcommand{\Im}{\Imim}
\newcommand{\R}{\mathbb{R}}

\newcommand{\Z}{\mathbb{Z}}
\def\Sph{\mathbb{S}}

\newcommand{\B}{\mathcal{B}}

\newcommand{\one}{{\mathbb{1}}}
\newcommand{\EQ}[1]{\begin{equation}  \begin{split} #1 \end{split} \end{equation} }
\def\embed{\hookrightarrow}
\def\const{\mathrm{const}\cdot}
\def\eps{\varepsilon}
\def\what{\widehat}
\def\mes{\mc{M}}
\def\calB{\mc{B}}
\def\calF{\mc{F}}
\def\calS{\mc{S}}
\def\calD{\mc{D}}
\def\nn{\nonumber}

\def\cB{\calB}

\numberwithin{equation}{section}

\begin{document}

\title[Structure formulas for wave operators]{Structure formulas for wave operators}

\author{M.\ Beceanu}
\thanks{The first author thanks the University of Chicago for its hospitality during the summers of 2015 and 2016.}
\address{Department of Mathematics, University at Albany, State University of New York, 1400 Washington Avenue, Albany, NY 12222 }
\email{mbeceanu@albany.edu}

\author{W.\ Schlag}
\thanks{The second author was partially supported by NSF grant DMS-1500696
during the preparation of this work.}
\address{Department of Mathematics, The University of Chicago, 5734 South University Avenue, 
Chicago, IL 60637}
\email{schlag@math.uchicago.edu}

\begin{abstract}
We revisit the structure formula for the intertwining wave operators $W_{\pm}$ associated with $H=-\Delta+V$ in $\R^{3}$ under suitable decay conditions on~$V$. We establish quantitative bounds on the structure function. Throughout we  assume that $0$ energy is  regular for $H=-\Delta+V$. 
\end{abstract}

\maketitle

\section{Introduction}
\lb{sec:intro}

In this paper we revisit the study of wave operators for $H=-\Delta+V$ in three dimensions. Define $B^{\beta}$, $\beta\ge0$, as the subspace of $L^{2}$ consisting of functions with the property that 
\[
\|f\|_{B^{\beta}}:= \| \one_{[|x|\le 1]} f\|_{2} + \sum_{j=0}^{\infty} 2^{j\beta} \big\| \one_{[2^{j}\le |x|\le 2^{j+1}]} f \big\|_{2} <\infty
\]
Then for $V\in B^{\beta}(\R^{3})$, $\beta\ge\frac12$, $V$ real-valued, the wave operators 
\[
W_{\pm} = \lim_{t\to\pm \infty} e^{itH} e^{-itH_{0}}
\]
exist in the usual strong $L^{2}$ sense, with $H_{0}=-\Delta$. They are isometries from $L^2$ onto the absolutely continuous spectral subspace of $H$ in $L^{2}$, and there is no singular continuous spectrum (asymptotic completeness), see Sections~\ref{sec:spectral}, \ref{sec:WOP}.  In a series of papers,  Yajima~\cite{yajima0,yajima,yajima6,yajima8}, established the $L^{p}$ boundedness of the wave operators assuming that zero energy is neither an eigenvalue nor a resonance (and later also obtained more restrictive  results of this nature if this condition fails). These results are very useful for nonlinear dispersive wave equations, since by the intertwining property $W_{\pm} f(H)P_{c}= f(H_{0})W_{\pm}$, $H_{0}:=-\Delta$, we may transfer Strichartz estimates from the free case $H_{0}$ to the perturbed evolution of $H$. 

The first author combined some of Yajima's formalism with his Wiener algebra technique~\cite{bec,becgol} to obtain a {\em structure formula} for the wave operators~\cite{bec1}.  In fact, he showed that the wave operators act on functions by the superposition of elementary operations. The paper~\cite{bec1} is not entirely accurate.  The first result of this paper is to present a corrected version of the structure formula. By $B^{1+}$ we mean $B^{\beta}$ for some $\beta>1$. In Section~\ref{sec:spectral} we will define what a zero energy resonance or eigenfunction means in this context. $\mc M$ denotes signed Borel measures.  
We also go beyond~\cite{bec1} by obtaining quantitative control of the structure function as well as stability estimates for it. 

\begin{theorem}\lb{thm:struct}
Assume that $V \in B^{1+}$ is real-valued and that $H = -\Delta+V$ admits no eigenfunction or resonance at zero energy. There exists  $g(x, y, \omega) \in L^1_\omega \mc M_y L^\infty_x$, i.e., 
\EQ{\lb{eq:gfin}
  \int_{\Sph^2} \|g(x, dy, \omega)\|_{\mc M_y L^{\infty}_x}  \dd \omega < \infty
} 
such that for $f \in L^2$ one has the representation formula
\be\begin{aligned}\lb{eqn1.8}
(W_{+} f)(x) &= f(x) + \int_{\Sph^2} \int_{\R^3} g(x, dy, \omega) f(S_\omega x - y)   \dd \omega.
\end{aligned}\ee
where $S_{\omega}x=x-2(x\cdot\omega)\omega$ is a reflection. A similar result holds for $W_{-}$. 

Let $X$ be any Banach space of measurable functions on $\R^3$ which is invariant under translations and reflections, and in which Schwartz functions are dense. Assume further that  $\|\one_H  f\|_X\le A\|f\|_X$ for all half spaces $H\subset\R^3$ and $f\in X$ with some uniform constant~$A$. Then 
\EQ{\label{eq:HW}
\| W_+ f\|_X \le AC(V) \|f\|_X \qquad\forall\; f\in X
}
where $C(V)$ is a constant depending on $V$ alone. 
\end{theorem}

The structure function $g(x, y, \omega)$ only depends on $x$ through the $1$-dimensional coordinate $x_\omega:=x \cdot \omega$, that is $g(x, y, \omega) \equiv \tilde g(x_\omega, y, \omega)$, and it has the additional regularity
\be\lb{reg}
\|\partial_{x_\omega} \tilde g(x_\omega, y, \omega)\|_{L^1_\omega \mc M_y \mc M_{x_\omega}} < \infty.
\ee
In particular, $W_\pm$ and $W_\pm^*$ are bounded on $X=L^p$, $1 \leq p \leq \infty$: if $f \in L^2 \cap L^p$, then
\be\lb{1.9}
\|W_{\pm} f\|_{L^p} + \|W_{\pm}^* f\|_{L^p} \les \|f\|_{L^p}.
\ee
This improves on Yajima's results since (i) less is required of~$V$, and (ii),   the class of Banach spaces~$X$ in~\eqref{eq:HW} is considerably more general than Lebesgue or Sobolev spaces.

 It is of course desirable to have a quantitative estimates on $g$ in place
of mere finiteness in~\eqref{eq:gfin}.  This is a somewhat delicate matter, since the so-called limiting absorption principle for the perturbed resolvent
are typically noneffective (see however~\cite{RT}). Clearly, any bound on $g$ will 
require a quantitative version of the zero energy condition.  In the following theorem we obtain a bound in terms of a global (in energy) bound 
on the Birman-Schwinger operator, see~\eqref{eq:Res Bd Intro}. The somewhat unusual $L^{\infty}\to L^{\infty}$ limiting-absorption principle is natural in this context.

\begin{theorem}\lb{thm:structQ}
Let $V\in B^{1+2\gamma}$, $0<\gamma$. 
Under the hypotheses of the previous theorem we have the following quantitative  bound on the structure function $g$, 
\EQ{\label{eq:gQ}
\int_{\Sph^2 }  \|  g(x,dy,\omega)\|_{\mes_{y} L^\infty_x}  \, d\omega & \le C_{0} (1+\|V\|_{B^{1+2\gamma}})^{38+\frac{105}{\gamma}} (1+M_{0})^{4+\frac{3}{\gamma}}
}
where 
\EQ{
\label{eq:Res Bd Intro}
\sup_{\eta\in\R^{3}} \sup_{\eps>0}\big \|   \big(I + R_0(|\eta|^2 \pm i\eps ) V  \big)^{-1} \big\|_{\infty \to \infty} =: M_{0}<\infty
} 
and $C_{0}$ is some absolute constant.   The right-hand side of~\eqref{eq:gQ} also controls $C(V)$ in~\eqref{eq:HW} as well as~\eqref{reg}.

Let $\tilde V \in B^{1+2\gamma}$ satisfy 
\EQ{\lb{eq:VVt}
\|  V-\tilde V\|_{B^{1+2\gamma}} < c_{0} \min(M_{0}^{-1},\|  V\|_{B^{1+2\gamma}} ),
} 
where $c_{0}\ll 1$ is some absolute constant.  Then $\tilde V$ obeys the $0$-energy condition, and there is the following {\em stability bound} on the structure functions $g, \tilde g$: 
\begin{align}
\label{eq:gtildeg}
& \int_{\Sph^2 }  \|  g(x,dy,\omega)-\tilde g(x,dy,\omega)\|_{\mes_{y} L^\infty_x}  \, d\omega \\
& \le C_{1}(\gamma, \|  V\|_{B^{1+2\gamma}} , M_{0}) \big[ \| V-\tilde V\|_{B^{1+2\gamma}}  + \| (V-\tilde V)/(|V|+|\tilde V|)\one_{[|V|+|\tilde V|\ne0]} \|_{\infty}\big] \nn
\end{align}
where $C_{1}$ is as on the right-hand side of~\eqref{eq:gQ}, albeit with different numbers in the exponents.  
\end{theorem} 

The appearance of the $L^{\infty}$-norm on the right-hand side of~\eqref{eq:gtildeg} is an unfortunate technical issue. 
The finiteness of $M_0$ will be shown in Section~\ref{sec:spectral}. It requires the $0$-energy assumption we impose on $H=-\Delta+V$. It would be interesting to have an {\em effective bound} on~$M_0$
in terms of the quantity 
\[
M_{00}:= \|   \big(I + (-\Delta)^{-1} V  \big)^{-1} \big\|_{\infty \to \infty}  <\infty
\]
which is precisely what our $0$-energy assumption amounts to, and quantifiable properties of~$V$, cf.~\cite{RT}. But we do not pursue these matters here, and therefore bound $g$ in terms of $M_0$, rather than $M_{00}$. 
Theorem~\ref{thm:structQ} is only interesting for large potentials. 
We remark that for $V$ small in $B^{1+}$ the $0$-energy condition automatically holds, $M_0\simeq 1$, and 
\EQ{\label{eq:gQklein}
\int_{\Sph^2 }  \|  g(x,dy,\omega)\|_{\mes_{y} L^\infty_x}  \, d\omega & \le C_{0} \|V\|_{B^{1+}}
}
This will be a byproduct of our results.


\section{Function spaces and interpolation}
\label{sec:Lpq}

For more background on the material in this section cf.~\cite[Appendices]{bec} and the comprehensive treatment in~\cite{bergh}. 
We recall  the family of Lorentz spaces $L^{p,q}(\R^d)$ defined as 
\EQ{
\|f\|_{L^{p,q}(\R^d)} &= \Big( \int_0^\infty \big(t^{\frac1p} f^*(t))^q\,\frac{dt}{t}\Big)^{\frac1q} \qquad 1\le q<\infty\\
\|f\|_{L^{p,\infty}(\R^d)} &=  \sup_{t\ge0} \; t^{\frac1p} f^*(t) \qquad   q=\infty
}
where $f^*$ is the nonincreasing rearrangement of $f$. 
The duality relations are $(L^{p,q})'=L^{p',q'}$ if $1<p<\infty$, $1\le q<\infty$.  Under real interpolation one has for $1\le p_0\ne p_1\le\infty$ 
\EQ{
\big(   L^{p_0,q_0}, L^{p_1,q_1}\big)_{\theta, r} = L^{p,r}
}
where $0<\theta<1$, $1\le q\le\infty$, and $\frac1p=\frac{1-\theta}{p_0} + \frac{\theta}{p_1}$.  In particular, 
\[
\big(   L^{p_0}, L^{p_1}\big)_{\theta, r} = L^{p,r}
\]
The Marcinkiewicz interpolation theorem in this setting states the following, see \cite[Theorem 5.3.2]{bergh}: 

\begin{theorem}\label{thm:Mar}
$T: L^{p_0,r_0} \to L^{q_0,s_0}$ and $T: L^{p_1,r_1} \to L^{q_1,s_1}$
with $p_0\ne p_1$ and $q_0\ne q_1$ implies that $T:L^{p,r} \to L^{q,s}$ if $\infty\ge s\ge r>0$ and $\frac1p=\frac{\theta}{p_0} + \frac{1-\theta}{p_1}$, 
$\frac1q=\frac{\theta}{q_0} + \frac{1-\theta}{q_1}$, $0<\theta<1$. 
\end{theorem}

One has H\"{o}lder's  inequalities, see \cite{oneil}, 
\EQ{\label{eq:Ho1}
\| fg\|_{L^{r,s}} &\le r' \|f\|_{L^{p_1,q_1}} \|g\|_{L^{p_2,q_2}}  \text{\ \ provided}\\
 \frac{1}{p_1}+\frac{1}{p_2} &= \frac1r<1,\; \frac{1}{q_1}+\frac{1}{q_2}\ge\frac1s
 }
 and the endpoint
 \EQ{\label{eq:Ho2}
\| fg\|_{L^{1}} &\le \|f\|_{L^{p,q_1}} \|g\|_{L^{p',q_2}}, \quad \frac{1}{q_1}+\frac{1}{q_2}\ge 1
}
as well as Young's inequalities:
\EQ{\label{eq:Y1}
\| f\ast g\|_{L^{r,s}} &\le 3r \|f\|_{L^{p_1,q_1}} \|g\|_{L^{p_2,q_2}} \text{\ \ provided} \\
\frac{1}{p_1}+\frac{1}{p_2} &= \frac1r+1>1,\;\frac{1}{q_1}+\frac{1}{q_2}\ge\frac1s
}
and the endpoint
\EQ{\label{eq:Y2}
\| f\ast g\|_{L^{\infty}} &\le \|f\|_{L^{p,q_1}} \|g\|_{L^{p',q_2}}, \quad \frac{1}{q_1}+\frac{1}{q_2}\ge 1
}
For any Banach space $X$ one defines the vector-valued space $\dot\ell^{s}_q(X)$ and $ \ell^{s}_q(X)$ as
\EQ{\label{eq:ellsq}
\| \{ f_j\}_j \|_{\dot\ell^{s}_q(X)} &= \Big( \sum_{j\in\Z} \Big( 2^{js} \|f_j\|_X\Big)^q \Big)^{\frac1q} \\
\| \{ f_j\}_j \|_{\ell^{s}_q(X)} &= \Big( \sum_{j\ge0} \Big( 2^{js} \|f_j\|_X\Big)^q \Big)^{\frac1q} 
}
where $s\in\R$ and $1\le q<\infty$. The usual modification needs to made for $q=\infty$. 

Then, see \cite[Theorem 5.6.1]{bergh}, 
\EQ{\label{eq:vec int}
\Big( \ell^{s_0}_{q_0}(X), \ell^{s_1}_{q_1}(X)\Big)_{\theta,q} = \ell^s_q(X)
}
where $s_0\ne s_1$, $0<q_0,q_1\le\infty$, $0<\theta<1$, $s=s_0(1-\theta)+s_1\theta$, $0<q\le \infty$, $\frac{1}{q}=\frac{1-\theta}{q_{0}}+\frac{\theta}{q_{1}}$, and the same holds for the dotted spaces. 

\begin{definition}\label{def:Bspace}
Let $A_k:=\{ x\in \R^d\:|\: 2^k\le |x|\le 2^{k+1} \}$ for each $k\in \Z$. 
For any $\alpha\in \R$ we set
\EQ{
\dot B^\alpha &:= \Big\{ v\in L^2_{\text{loc}} (\R^d\setminus\{0\})\:|\: \sum_{k\in\Z} 2^{\alpha k} \|  \one_{A_k} v\|_2 <\infty\Big\}\\
B^\alpha &:= \Big\{ v\in L^2_{\text{loc}} (\R^d )\:|\: \| \one_{B(0,1)} v\|_2+\sum_{k\ge0 } 2^{\alpha k} \|  \one_{A_k} v\|_2 <\infty\Big\}
}
where the sums on the right-hand side are the respective norms. 
\end{definition} 

In the notation of \eqref{eq:ellsq} one has 
\EQ{
\|v\|_{\dot B^\alpha} &= \Big\| \{ \one_{A_k} v\}_{k\in\Z} \Big\|_{\dot\ell^{\alpha}_1(L^2(\R^d))}=:\| \iota(v)\|_{\dot\ell^{\alpha}_1(L^2(\R^d))}  \\
 \|v\|_{B^\alpha} &= \Big\| \{ \one_{A_k} v\}_{k\ge0}\cup\{\one_{B(0,1)}v\} \Big\|_{\ell^{\alpha}_1(L^2(\R^d))}=:\| \iota_0(v)\|_{\ell^{\alpha}_1(L^2(\R^d))}
}
This allows us to use  interpolation as in \eqref{eq:vec int}. 

\begin{lemma}
\label{lem:Balpha interpol}
The spaces in Definition \ref{def:Bspace} satisfy the following properties:
\begin{itemize}
\item  $\dot B^{2\alpha} =(L^2, |x|^{-2} L^2)_{\alpha,1}$ and $B^{2\alpha} =(L^2, \langle x\rangle^{-2} L^2)_{\alpha,1}$   for $0<\alpha<1$. 
\item 
$\dot B^{\alpha} = (\dot B^{\alpha_0},   \dot B^{\alpha_1})_{\theta,1}$ where $\alpha=\theta\alpha_1+(1-\theta)\alpha_0$ where $0<\theta<1$ and $\alpha_0\ne\alpha_1$. 
The same holds for the undotted spaces. 
\item 
The embeddings $\dot B^{\frac12} \embed  L^{\frac32,1}(\R^3)$ and $\dot B^{1} \embed  L^{\frac65,1}(\R^3)$ hold, and the same is true for the undotted spaces. 
By duality, $L^{3,\infty}(\R^3)\embed (B^{\frac12} )'$.
\end{itemize}
\end{lemma}
\begin{proof}
One has $\|v\|_2= \| \iota(v)\|_{\dot\ell^{0}_2(L^2(\R^d))}$ and 
\[
\| v\|_{ |x|^{-2} L^2} \simeq \| \iota(v)\|_{\dot\ell^{2}_2(L^2(\R^d))},\quad \| v\|_{ \langle x\rangle^{-2} L^2} \simeq \| \iota_0(v)\|_{\ell^{2}_2(L^2(\R^d))}
\]
Next, $\dot B^0\embed L^2$, $B^0\embed L^2$ and $\dot B^{\frac32} \embed L^1$, $B^{\frac32} \embed L^1$ (by H\"older). Therefore, 
\EQ{\label{eq:BdotL}
 \dot B^{\frac32\theta} = (\dot B^0, \dot B^{\frac32})_{\theta,1} \embed (L^2, L^1)_{\theta,1} = L^{p,1}
}
where $p=\frac{2}{1+\theta}$, $0<\theta<1$. Similarly, $B^{\frac32\theta}\embed L^{p,1}$. 
\end{proof}


\section{Spectral theory} 
\label{sec:spectral}

This section discusses zero energy eigenvalues and resonances, as well as embedded resonances.  These are classical questions, but due to the somewhat special class of potentials under consideration we supply the details.  See also~\cite{bec}. 

\begin{definition}\lb{def:0res}
Let $V\in L^{\frac32,1}(\R^{3})$. We say that $0$ energy is regular for $H=-\Delta+V$ if 
\be\lb{1.10}
f = - R_0(0) V f
\ee
has no solution $f \in L^{\infty}$, $f\ne0$. 
\end{definition}

It is standard that $H=-\Delta+V$ is a self-adjoint operator on $L^{2}(\R^{3})$ if $V\in L^{3/2, 1}$ is real-valued; for example, 
such potentials obey the Rollnick condition, cf.~\cite[Proposition 2.9]{bec}.  
The following lemma relates the $0$ energy criterion of Definition~\ref{def:0res} with the more common definition involving weighted $L^2$-spaces, see~\cite{jenkat}, \cite{ionschlag}.  For this we will assume more than $V \in L^{3/2, 1}$, namely $V\in B^{\frac12}$ which is a smaller space by Lemma~\ref{lem:Balpha interpol}. 

\begin{lemma}\lb{lem2.1} 
Let  $V \in B^{\frac12}$. If (\ref{1.10}) holds with $f \in L^{\infty}$, then   
$f \in  (B^{\frac12})'$, and conversely. 
Note that $f\in (B^{\frac12})'$ if and only if 
$$
\sup_{k \ge 0} \; 2^{-k/2} \|\one_{A_k} f\|_{L^2} + \| \one_{B(0,1)}f\|_2 < \infty.
$$
\end{lemma}

\begin{proof} 
The kernel of $R_0(0) = (-\Delta)^{-1}$  is explicitly given by $R_0(0)(x, y) = \frac{1}{4\pi} \frac{1}{|x-y|}$. Note that $|x|^{-1}\in L^{3,\infty}$. 
Write $V = V_1 + V_2$, where $V_1$ is bounded of compact support (and therefore in $L^1$) and $\|V_2\|_{L^{3/2, 1}}$ is small. Then 
\[
\| R_0(0) V_2 \|_{L^{3, \infty}\to L^{3, \infty}} <1
\]
by \eqref{eq:Ho2} and \eqref{eq:Y1}. 
Thus $(I + R_0(0) V_2)^{-1}$ is a bounded operator on $L^{3, \infty}$ and
$$
f = -(I + R_0(0) V_2)^{-1} R_0(0) V_1 f.
$$
Since $V_1 f \in L^1$, $R_0(0) V_1 f \in L^{3, \infty}$ by \eqref{eq:Y1}, so by Lemma~\ref{lem:Balpha interpol}, 
$$f \in L^{3, \infty} \embed  (B^{\frac12})'.$$
 
Now assume $f \in (B^{1/2})'$, then $V f \in L^1$, so $f = R_0(0) V f \in L^{3, \infty}$. With the previous splitting $V = V_1 + V_2$,  write
\EQ{\label{eq:V1V2}
f = -(I + R_0(0) V_2)^{-1} R_0(0) V_1 f.
}
Here we used that 
\[
\| R_0(0) V_2 \|_{L^{\infty}\to L^{\infty}} <1
\]
One has $V_1 f \in L^{3/2, 1}$ since $V_1f\in L^2$ with compact support, $R_0(0) V_1 f \in L^{\infty}$, so $f \in L^{\infty}$.
\end{proof}

We write the resolvent identity as
\be\lb{2.29}
(I + R_0(z) V)^{-1} = I - R_V(z) V;\ R_V(z) = (I + R_0(z) V)^{-1} R_0(z)
\ee
for $\Im z>0$. Here $R_{V}(z)=(H-z)^{-1}$. The following lemma addresses the limits $\Im z\to0+$. 
In particular, we  obtain a {\em limiting absorption principle} on $L^{\infty}$, see~\cite{ionschlag,bec} for more background.  By (\ref{2.29}), $I + R_0 V$ is invertible on $L^{\infty}$ if and only if $R_V$ is
 bounded from $L^{3/2, 1}$ to $L^{\infty}$.

\begin{lemma}\lb{lemma2.3} 
Assume that $V \in L^{3/2, 1}$ is real-valued. 
Then $I + R_0(|\eta|^2 + i0)V$ is invertible on $L^{\infty}$ for any $\eta \ne 0$. For $\eta = 0$ it is invertible if and only if  zero energy 
is regular in the sense of Definition~\ref{def:0res}.
In that case, $I + R_0(|\eta|^2 \pm i\eps ) V$ is also invertible on $L^{\infty}$ and its inverse is bounded in $\B(L^{\infty}, L^{\infty})$, uniformly  for $\eps  \geq 0$ and~$\eta$.
\end{lemma}
\begin{proof}
Let $C_0$ be the space of continuous functions that vanish at infinity.  It follows from~\eqref{eq:Y2} 
that  $R_0(|\eta|^2 + i0) V: L^{\infty}\to C_0$.  Indeed, 
\[
\lim_{y\to0} \| (Vf)(\cdot + y) -Vf\|_{L^{3/2, 1}} =0
\]
gives the continuity of $R_0(|\eta|^2 + i0) Vf$ for any $f\in L^\infty(\R^3)$. Second, if $\eps>0$ then there exists $\tilde V\in C_{\mathrm{comp}}$ with $\|V-\tilde V\|_{L^{3/2, 1}}<\eps$. 
Clearly, $$(R_0(|\eta|^2 + i0) \tilde Vf)(x)\to 0$$  as $|x|\to\infty$. By \eqref{eq:Y2}, we obtain the same for $V$ in place of $\tilde V$. 

 By  Arzel\`{a}-Ascoli, a set $A\subset C_0$ is precompact in $C_0$ if and only if one has
\begin{itemize}
\item  equicontinuity:
\be\lb{1st_cond}
\forall \eps >0\ \exists \delta>0\ \forall |y|<\delta\ \forall a \in A\qquad \|a(\cdot -y) - a\|_{\infty} < \eps ;
\ee
\item  uniform decay at infinity:
\be\lb{2nd_cond}
\forall \eps >0\ \exists R\ \forall a \in A\qquad\qquad  \|\one_{|x|>R}(x) a(x)\|_{\infty} < \eps .
\ee
\end{itemize}
We wish to verify these conditions for  $A=R_0(|\eta|^2 + i0) V(\calB)$ where $\calB \subset L^{\infty}$ is bounded; in fact, we may take it to be the unit ball. It suffices to assume 
 that $V$ is continuous and compactly supported. Indeed, approximating $V$ in $ L^{3/2, 1}$ by such potentials and using \eqref{eq:Y2} as above implies the general case. 
 For the uniform vanishing at $\infty$, suppose that $V(x)=0$ if $|x|\ge M$. Then 
 \EQ{
 |R_0(|\eta|^2 + i0) Vf(x)|\les \|Vf\|_{1} (|x|-M)^{-1} \les \|V\|_1 (|x|-M)^{-1}
 }
 and the vanishing follows. For the equicontinuity we introduce for any $\lambda\in\R$ the kernels
 \[
k_{1,\lambda} (x) = \one_{B(0,1)}(x)\frac{e^{i\lambda |x|}}{|x|}  ,\quad k_{2,\lambda} (x) = \one_{B(0,1)^c}(x)\frac{e^{i\lambda |x|}}{|x|} 
 \]
Then for any $|y|\le \frac12$ with absolute constants $C$, 
\EQ{
\| k_{1,\lambda} (\cdot+y)- k_{1,\lambda}(\cdot)\|_{L^1} + \| k_{2,\lambda} (\cdot+y)- k_{2,\lambda}(\cdot)\|_{L^4} &\le C(1+|\lambda|)|y| 
}
 It then follows that  (with $\lambda=|\eta|$ and using $\|f\|_{\infty}\le1$) 
  \EQ{
 & \|(R_0(|\eta|^2 + i0) Vf)(\cdot+y) - (R_0(|\eta|^2 + i0) Vf)(\cdot) \|_\infty\\ &\le C\langle \eta\rangle\, |y| (  \|V\|_\infty + \|V\|_{\frac43})
  }
 Since the right-hand side does not depend on $f$, equicontinuity holds. Thus, $R_0(|\eta|^2 + i0) V:L^\infty\to L^\infty$ is a compact operator, and so is  $VR_0(|\eta|^2 - i0):L^1\to L^1$ (as the former is the adjoint of the latter). 
 
By Fredholm's alternative in Banach spaces, $I + R_0(|\eta|^2 + i0) V$ is invertible in $L^\infty$ if and only if the equation
\EQ{\label{eq:kernel}
f = -V R_0(|\eta|^2 - i0) f
}
has no other than the trivial solution in  $L^1$. Let $f$ solve \eqref{eq:kernel}. Then $$g = R_0(|\eta|^2 - i0) f\in L^{3, \infty}$$ and satisfies the equation
$$
g= -R_0(|\eta|^2 - i0) V g.
$$
By the same argument as in~\eqref{eq:V1V2} it follows that $g\in L^\infty(\R^3)$.  Assume at first that $\eta \ne 0$. Since $V$ is real-valued,   $\langle g, V g\rangle\in \R$ whence $$\Im \langle V g, R_0(|\eta|^2-i0) V g \rangle = 0.$$ Hence $$\widehat {Vg} \mid_{|\eta| \set \Sph^2 } = 0.$$ This is well-defined since $Vg\in L^1$.  By \cite[Proposition 2.4]{golsch} $$g=-R_0(|\eta|^2-i0) Vg \in \langle x \rangle^{1/2-\eps } L^2$$ for some $\eps >0$. Since $g$ is a distributional solution of the equation $$(-\Delta +V - |\eta|^2) g = 0$$ and $g \in L^\infty$, it follows that $$g \in \langle\nabla\rangle^{-2} L^{3/2, 1}_{\mathrm{loc} }\subset H^1_{\mathrm{loc}}.$$ By~\cite[Theorem 2.1]{ionjer}, for $\eta \ne 0$ we conclude from the preceding that $g=0$.  It then follows that $f=(-\Delta-|\eta|^2)g=0$ (distributionally), and so \eqref{eq:kernel} only has the trivial solution.

If $\eta=0$, we refer to Definition~\ref{def:0res} and to Lemma \ref{lem2.1}.  To be specific, here too, $I + R_0(0) V$ is invertible in $L^\infty$ if and only if the equation
\EQ{\label{eq:kernel'}
f = -V R_0(0) f
}
 has no other than the trivial solution in  $L^1$. But by the same argument as before $g=  R_0( 0) f\in L^\infty$ solves $g= -R_0(0) V g$. Definition~\ref{def:0res} then requires that $g=0$ and therefore also $f=0$. 
In summary, the inverse $(I + R_0(|\eta|^2 + i0) V)^{-1}$ exists for every $\eta \in \set R^3$.

The map $\lambda \mapsto R_0(\lambda^2 + i0) V \in \B(L^{\infty}, L^{\infty})$ is continuous, and the inverses have uniformly bounded norms when $\lambda$ is in a compact set.
 By Lemma~\ref{lem:Asq} $$\| (R_0(\lambda^2+i0)V)^2 \|_{\infty\to\infty}\to 0$$ as $\lambda\to\infty$.  Therefore,
 \EQ{\lb{eq:RV2}
 (I+R_0(\lambda^2+i0)V)^{-1}= (1-R_0(\lambda^2+i0)V)(I- (R_0(\lambda^2+i0)V)^2)^{-1} 
 }
 is uniformly bounded as operators on $L^\infty$ for all  $|\lambda|\gg1$.

This extends to any set in the complex plane at a positive distance away from the eigenvalues --- in particular to the whole right half-plane.
\end{proof}

It is easy to see that $\|R_0(\lambda^2+i0)V\|_{\infty\to\infty}$ does not depend on $\lambda$. In fact,
\[
(R_0(\lambda^2+i0)Vf)(0)= \frac{1}{4\pi} \int_{\R^3} \frac{e^{i\lambda|y|}}{|y|} f(y)\, dy = (-\Delta)^{-1} (e^{i\lambda|\cdot|} f)(0)
\]
and therefore
\[
\|R_0(\lambda^2+i0)V\|_{\infty\to\infty} = \|(-\Delta)^{-1}V\|_{\infty\to\infty}
\]
which does not decay in $\lambda$. 
To circumvent this issue, one can square the operator as in \eqref{eq:RV2}.

\begin{lemma}\label{lem:Asq}
For $V\in L^{\frac32,1}(\R^3)$ we have 
\[
\| (R_0(\lambda^2+i0)V)^2 \|_{\infty\to\infty}\to 0 \text{\ \  as\ \ } \lambda\to\infty
\]
\end{lemma}
\begin{proof}
In view of \eqref{eq:Y2} we may reduce ourselves to the case of a smooth, compactly supported $V$. Then 
\EQ{
(R_0(\lambda^2+i0)V)^2 f(x) =   \int_{\R^3} K_\lambda(x,y) V(y) f(y)\, dy
}
where 
\EQ{
K_\lambda(x,y):=  \frac{1}{16\pi^2} \int_{\R^3} \frac{e^{i\lambda(|x-u|+|u-y|)}}{|x-u|\, |u-y|} V(u) \,du 
}
We claim that 
\EQ{\lb{eq:clLft}
\sup_{x,y\in\R^6} | K_\lambda (x,y) | \to 0 \text{\ \  as\ \ } \lambda\to\infty
}
If so, then $\|f\|_\infty \le 1$ implies that 
\[
\| (R_0(\lambda^2+i0)V)^2 f\|_\infty \le  \|K_\lambda (x,y)\|_{L^\infty_{x,y}} \|V\|_1  \to 0 \text{\ \  as\ \ } \lambda\to\infty
\]
Given $\delta>0$ we let $\chi$ be a smooth radial bump function in $\R^3$ with $\chi=0$ on the unit ball and $\chi(u)=1$ if $|u|\ge2$.  Then 
\EQ{\lb{eq:Kbreak}
&K_\lambda(x,y)\\&=  \frac{1}{16\pi^2} \int_{\R^3} \frac{e^{i\lambda(|x-u|+|u-y|)}}{|x-u|\, |u-y|} \chi(|u-x|/\delta)\chi(|u-y|/\delta) V(u) \,du  \\
& + \frac{1}{16\pi^2} \int_{\R^3} \frac{e^{i\lambda(|x-u|+|u-y|)}}{|x-u|\, |u-y|} \big[ 1- \chi(|u-x|/\delta)\chi(|u-y|/\delta)\big] V(u) \,du
}
The second line contributes at most $O(\delta)$ to $\|K_\lambda (x,y)\|_{L^\infty_{x,y}}$ in  \eqref{eq:clLft}.
Fix some small $\delta>0$.  We integrate by parts  in $u$ in the first line of~\eqref{eq:Kbreak} using that 
\[
(i\lambda)^{-1} \vec v(x,y,u)\cdot \nabla_u e^{i\lambda(|x-u|+|u-y|)} = |\vec v(x,y,u)|^2\, e^{i\lambda(|x-u|+|u-y|)} 
\]
with 
\[
\vec v(x,y,u) = \frac{u-x}{|u-x|} +  \frac{u-y}{|u-y|}
\]
The degenerate case where $\vec v(x,y,u)=0$ occurs if $u$ lies on the line segment joining the points $x$ and $y$. 
This however contributes nothing to the integral in $u$. Similarly, of $\vec v$ is small, then that will contribute very little to the integral.  Thus, introduce a cut-off function $\chi(\vec v(x,y,u)/\eps)$  into the first integral in \eqref{eq:Kbreak}, which we denote by $\tilde K_\lambda(x,y)$ :
\EQ{\lb{eq:Kbreak2}
&\tilde K_\lambda(x,y)\\&=  \frac{1}{16\pi^2} \int_{\R^3} \frac{e^{i\lambda(|x-u|+|u-y|)}}{|x-u|\, |u-y|} \chi(|u-x|/\delta)\chi(|u-y|/\delta) \chi(\vec v(x,y,u)/\eps) V(u) \,du  \\
& + \frac{1}{16\pi^2} \int_{\R^3} \frac{e^{i\lambda(|x-u|+|u-y|)}}{|x-u|\, |u-y|} (1-\chi(\vec v(x,y,u)/\eps))  \chi(|u-x|/\delta)\chi(|u-y|/\delta) V(u) \,du
}
The second line here contributes $o(1)$ to $\|\tilde K_\lambda(x,y)\|_{L^\infty_{x,y}}$ as $\eps\to0$, whereas in the first line we integrate by parts (with $\eps>0$ small but fixed) using the operator  
\[
\calL := |\vec v(x,y,u)|^{-2} (i\lambda)^{-1} \vec v(x,y,u)\cdot \nabla_u
\]
Sending $\lambda\to\infty$ then shows that this contributes $o(1)$ to $\|\tilde K_\lambda(x,y)\|_{L^\infty_{x,y}}$.  Note that the separation $\delta>0$ avoids the degeneracies arising here from $x,y$ coming too close to $u$. 
\end{proof} 

To summarize, we have shown that if $V \in L^{3/2, 1}$ is real-valued, and $0$ energy is regular as specified in Definition~\ref{def:0res}, then 
\EQ{
\label{eq:Res Bd}
\sup_{\eta\in\R^{3}} \sup_{\eps>0}\big \|   \big(I + R_0(|\eta|^2 \pm i\eps ) V  \big)^{-1} \big\|_{\infty \to \infty} =: M_{0}<\infty
}


\section{Existence and properties of wave operators}
\lb{sec:WOP}

Let $V$ be as in Definition~\ref{def:0res}, and real-valued. 
The orthogonal projection $P_p:L^{2}\to L^{2}$ onto the point spectrum is a finite-rank operator of the form
\be
P_p = \sum_{\ell=1}^N \langle \cdot, f_{\ell} \rangle f_{\ell}.
\ee
where $f_{\ell}$ are an orthonormal family of eigenfunctions of $H=-\Delta+V$ with eigenvalues $H f_\ell = \lambda_\ell f_\ell$.
Since $V\in L^{\frac32}$ obeys the Rollnick condition, cf.~\cite[(2.63)]{bec}, the Birman-Schwinger operator is Hilbert-Schmidt and  $N<\infty$. We are also assuming that there are no zero energy eigenfunctions (or a resonance).  The projection  $P_c=I-P_{p}$ is the orthogonal projection on the subspace corresponding to the continuous spectrum. Lemma~\ref{lemma2.3}  implies that the continuous spectrum $[0,\infty)$ of $H$ is purely absolutely continuous,\
see \cite[Theorem XIII.19]{reesim}. Thus the entire $L^{2}$ spectrum of $H$ consists of finitely many negative  eigenvalues (counted with multiplicity)  and the absolutely continuous spectrum $[0,\infty)$.  For ``nice'' potentials, Agmon's estimate shows that the eigenfunctions $f_{\ell}$ decay exponentially in the point-wise sense. We have no need for this strong property, and the following lemma will suffice. 

\begin{lemma}
\label{lem:felldecay}
Let $V\in L^{\frac32,1}(\R^{3})$ and suppose $f\in L^{2}$ solves $Hf = -E^{2}f$ with $E>0$  in the sense of tempered distributions. Then
$f\in (L^{1}\cap L^{\infty})(\R^{3})$. 
\end{lemma}
\begin{proof}
Since $R:= (-\Delta+E^{2})^{-1}$ takes the Schwartz space to itself, it follows that 
$
f=-RVf
$ in the sense of distributions. Splitting $V=V_{1}+V_{2}$ as before with $\|V_{1}\|_{L^{\frac32,1}}\ll 1$ and $V_{2}$ continuous with bounded support we also have 
\EQ{\label{eq:RV1}
(I+RV_{1})f = -RV_{2}f
}
Now $RV_{1}:L^{\infty}\to L^{\infty}$ with small norm and $V_{2}f\in L^{1}\cap L^{2}\subset L^{\frac32,1}$.
Hence $RV_{2}f\in L^{\infty}$ and $f\in L^{\infty}$, by inverting the operator on the left-hand side of~\eqref{eq:RV1}. 

So $f\in L^{2}\cap L^{\infty}\subset L^{3,\infty}$ and $Vf\in L^{1}(\R^{3})$. The convolution $RVf \in L^{1}$ by Young's inequality, and finally $f\in L^{1}$ as desired. 
\end{proof}

Next, we discuss the existence of the wave operators by the standard Cook's method. However, the class of potentials we consider requires more sophisticated estimates to make Cook's method work, namely the Keel-Tao endpoint~\cite{keetao}. 
Lemma~\ref{lem:cook} was shown in \cite{bec} to also hold when $V$ is in $L^{\frac32, \infty}_0$ (the closure of $L^{\frac32}$ in $L^{\frac32, \infty}$). 

\begin{lemma}
\label{lem:cook} 
 Let  $H=-\Delta+V$ be self-adjoint as in Definition~\ref{def:0res}.  Let $P_c$ be the projection on the continuous spectrum of $H$. Then 
 \EQ{
 W_+ = \slim_{t \to \infty} e^{itH} P_c e^{-itH_0} = \slim_{t \to \infty} e^{itH} e^{-itH_0}
 } 
 exists and  is an isometry on $L^2$. One has  $P_{c}=P_{\mathrm{a.c.}}= W_{+} W_{+}^{*}$. 
 Moreover, for any $f\in L^{2}(\R^{3})$ the integrals 
\be\begin{aligned}\label{eq:1.6}
W_+ f &= f + i \int_0^{\infty} e^{itH} V e^{-itH_0} f \dd t \\
&= P_c f + i \int_0^{\infty} e^{itH} P_c V e^{-itH_0} f \dd t.
\end{aligned}\ee
converge in the strong sense. 
There exist similar formulae for $W_-$ and $W_{\pm}^*$; in particular,
\be\lb{eqn1.7}
W_-^* f = f + i \int_{-\infty}^0 e^{itH_0} V e^{-itH} f \dd t.
\ee
\end{lemma}
\begin{proof}
The Strichartz estimates 
\EQ{\label{eq:keeltao}
\|e^{it H_0} f\|_{L^2_t L^{6, 2}_x} &\les \|f\|_{L^2}\\
\Big\|\int_\R e^{-isH_0} F(s) \dd s\Big\|_{L^2_x} &\les\|F\|_{L^2_t L^{6/5, 2}_x}, 
}
are the standard Keel--Tao endpoint~\cite{keetao} for the Schr\"{o}dinger evolution of $H_{0}=-\Delta$. They also hold for $e^{itH} P_c$, see \cite{bec}. 

Taking the time derivative of the left-hand side and integrating we obtain
\EQ{\lb{eq:cookt}
e^{itH} P_c e^{-itH_0} f = P_c f + i \int_0^t e^{isH} P_c V e^{-isH_0} f \dd s.
}
Note that by H\"{o}lder's inequality \eqref{eq:Ho1} one has $V:L^{6, 2} \to L^{6/5, 2}$ as a multiplication operator; in fact, this only requires $V\in L^{3,\infty}(\R^{3})$. 
Hence, by \eqref{eq:keeltao}  the integral in \eqref{eq:cookt} converges in norm and  we can send $t \to \infty$ and obtain the statement in (\ref{eq:1.6}) involving~$P_{c}$. Thus, endpoint Strichartz estimates imply the existence of the strong limit $\slim_{t \to \infty} e^{-itH} P_c e^{itH_0}$ in $L^2$.

We claim that 
\be\lb{28}
\lim_{t \to \infty} P_p \; e^{-itH_0} f = 0
\ee
for all $f \in L^2$. Indeed, since $L^{1}\cap L^{2}$ is dense in $L^{2}$ we may assume that $e^{itH_0} f$ decays like $|t|^{-3/2}$ in $L^{\infty}$.
By Lemma~\ref{lem:felldecay}  the pairing with $f_{\ell}\in L^{1}$ therefore decays as desired. 

Consequently,
$$
W_+ = \slim_{t \to \infty} e^{itH} e^{-itH_0} = \slim_{t \to \infty} e^{itH} P_c e^{-itH_0}.
$$
To obtain the integral representation without the projection $P_{c}$ we note that =
$$
\slim_{t \to \infty} i \int_0^t e^{isH} P_p V e^{-isH_0} \dd s = \slim_{t \to \infty} P_p (e^{itH} e^{-itH_0}-I) = -P_p.
$$
The final limit here is obtained from the time-decay of $e^{-itH_0}f$  in $L^{\infty}$ for $f\in L^{1}\cap L^{2}$ and the fact that $f_{\ell}\in L^{1}$. 
 The  relation $P_{c}=P_{\mathrm{a.c.}}= W_{+} W_{+}^{*}$ follows from the general principle that $W_{+} W_{+}^{*}=P_{\Ran(W_{+})}$ for isometries. 
\end{proof}

Expanding the right-hand side of \eqref{eq:1.6} iteratively by means of the Duhamel formula we obtain the formal expansion
\begin{align}\lb{Dyson}
W_+ f &= f + W_{1+} f + \ldots + W_{n+} f + \ldots,\\
\nonumber W_{1+} f &= i \int_{t>0} e^{-i t \Delta} V e^{i t \Delta} f \dd t,\ \ldots \\
W_{n+} f &= i^n \int_{t>s_1>\ldots>s_{n-1}>0} e^{-i(t-s_1)\Delta} V e^{-i(s_1-s_2) \Delta} V \ldots \\
\nonumber &e^{-i s_{n-1} \Delta} V e^{it\Delta} f \dd t \dd s_1 \ldots \dd s_{n-1}
\end{align}
 for $f\in L^2$. 
The first term is the identity, hence always bounded. Yajima~\cite{yajima0}  proved that each remaining term $W_{n+}$, $n \geq 1$, is bounded as an $L^p$ operator. And the operator norm grows exponentially with $n$: in $\R^3$
\be\lb{1.7}
\|W_{n+} f\|_{L^p} \le C^n \|V\|_{\langle x \rangle^{-1-\eps } L^2}^n \|f\|_{L^p}.
\ee
Thus, for small potentials, i.e., when $\|V\|_{\langle x \rangle^{-1-\eps } L^2} \ll 1$, Weierstra{\ss}'s criterion shows that (\ref{Dyson}) is summable, whence the full wave operators $W_{\pm}$ are    
$L^p$-bounded. In general, however,  the asymptotic expansion (\ref{Dyson}) may diverge.

In order to overcome this difficulty, for large $V$ Yajima \cite{yajima0} estimated a finite number of terms directly by this method. He used a separate argument to show the boundedness of the remainder, for which he had to assume that $V$ decays faster than~$\langle x \rangle^{-5-\eps }$.
In this paper, we avoid summing \eqref{Dyson} altogether and rely instead on the first author's Wiener algebra approach~\cite{bec,becgol}.

\begin{definition}
For $\eps >0$ we introduce the regularized operators
\be\begin{aligned}\lb{eq2.5}
W_{n+}^\eps  f&:= i^n \int_{0\leq t_1 \leq \ldots \leq t_n} e^{i(t_n-t_{n-1})H_0-\eps (t_n-t_{n-1})} V \ldots \\
& e^{i(t_2 - t_1)H_0 - \eps (t_2 - t_1)} V e^{it_1 H_0 - \eps  t_1} V e^{-it_nH_0} f \dd t_1 \ldots \dd t_n,
\end{aligned}\ee
together with
\be\lb{eq2.6}
W_+^\eps  = I + i \int_0^{\infty} e^{it H-\eps  t} V e^{-it H_0} \dd t.
\ee
\end{definition}

These regularizations behave as expected under the limit $\eps\to0$. 

\begin{lemma}\lb{lem:Weps}
 $W_+^\eps  f \to W_+ f$ strongly as $\eps  \to 0$ for each $f \in L^2$. Similarly, $W_{n+}^\eps  f \to W_{n+}f$ for each $n\ge1$. 
\end{lemma}
\begin{proof} 
It obviously follows from the  Strichartz estimates \eqref{eq:keeltao} that 
(with $t\ge0$) 
\EQ{ 
\sup_{\eps\ge0 }\|e^{it H_0 -\eps t} f\|_{L^2_t L^{6, 2}_x} &\les \|f\|_{L^2}\\
\Big\|\int_\R e^{isH}P_{c} F(s) \dd s\Big\|_{L^2_x} &\les\|F\|_{L^2_t L^{6/5, 2}_x}, 
}
Hence the tails of the integrals in \eqref{eq2.6} (under the projection $P_{c}$)  are uniformly small in $\eps>0$. 
On any compact interval $[0,T]$ we can pass to the limit $\eps\to0$ under the integral by dominated convergence. 

It remains to verify that for any $f \in L^2$
$$
\lim_{\eps  \downarrow 0} \int_0^{\infty} e^{itH - \eps  t} P_p V e^{-itH_0} f \dd t = \int_0^{\infty} e^{itH} P_p V e^{-itH_0} f \dd t,
$$
Since each side is a bounded operator on $L^{2}$ uniformly in $\eps>0$, it suffices to verify this for  $f\in L^{2}\cap L^{\infty}$. 
Specializing to a single eigenfunction $f_{\ell}$, 
\[
\lim_{\eps  \downarrow 0} \int_0^{\infty} e^{it\lambda_{\ell} - \eps  t} \langle  f , e^{itH_0} V f_{\ell}\rangle \dd t = \int_0^{\infty} e^{it\lambda_{\ell} } \langle  f , e^{itH_0} V f_{\ell}\rangle \dd t
\]
Since $Vf_{\ell}\in L^{1}$, 
\[
| \langle  f , e^{itH_0} V f_{\ell}\rangle |\les t^{-\frac32} \quad \forall\; t\ge 1
\]
whence the tails in these integrals are again uniformly small in $\eps$. On compact time intervals we may pass to the limit $\eps\to0$.  In summary,  $W_+^\eps  f \to W_+ f$ strongly as $\eps  \to 0$. 
The argument for $W_{n+}^\eps  f \to W_{n+}f$ is similar and we leave it to the reader. 
 \end{proof}
 
The operators $W_{n+}$ will be expressed in terms of the following kernels. 
Definition~\ref{def:kernels} is somewhat formal, but the subsequent lemmas will justify the formulas rigorously in the context of the wave operators. 

\begin{definition} \lb{def:kernels}
Let $V$ be a Schwartz function in $\R^3$. 
For $\eps  > 0$, let $T_{1\pm}^\eps (x_0, x_1, y)$ be defined in the sense of  distributions as
\be\lb{2.16}
(\mc F_{x_0}^{-1} \mc F_{x_1, y} T_{1\pm}^\eps )(\xi_0, \xi_1, \eta) := \frac {\widehat V(\xi_1 - \xi_0)}{|\xi_1 + \eta|^2 - |\eta|^2 \pm i\eps }
\ee
and, more generally, for any $n\ge1$
\be\begin{aligned}\lb{2.17}
(\mc F_{x_0}^{-1} \mc F_{x_n, y} T_{n\pm}^\eps )(\xi_0, \xi_n, \eta) &:=  
\int_{\R^{3(n-1)}} \frac{\prod_{\ell=1}^n \widehat V(\xi_{\ell} - \xi_{\ell-1}) \dd \xi_1 \ldots \dd \xi_{n-1}}{\prod_{\ell=1}^n (|\xi_{\ell}+\eta|^2-|\eta|^2 \pm  i \eps )}.
\end{aligned}\ee
Also let $T_\pm^\eps $ be given by the distributional Fourier transform 
\be\begin{aligned}\lb{eqn2.13}
\mc F_{y} T_\pm^\eps (x_0, x_1, \eta) &:= e^{ix_0 \eta} \big(R_V(|\eta|^2 \mp i\eps ) V\big)(x_0, x_1) e^{-ix_1\eta};
\end{aligned}\ee
where we assume that $0$ energy is regular for $H=-\Delta+V$; see Lemma~\ref{lem:0regBd} for a justification. 
\end{definition}

The right-hand sides of \eqref{2.16} and \eqref{2.17} are tempered distributions, whence the kernels $T_{n\pm}^\eps (\xi_0, \xi_n, \eta)$ are tempered distributions on $\R^9$.  In the following section we will find this kernel for $n=1$.   
Two variables are sufficient for representing $W_{n+}^\eps $, but a meaningful algebra structure requires one more variable. This is the reason for the presence of a third variable $x_0$ in (\ref{2.16}) and (\ref{2.17}).
For three-variable kernels $T(x_0, x_1, y)$ the expressions above suggests the following composition law $\oast$, which we define {\em formally}. 

\begin{definition}
\lb{def:comp}  
We formally compose three variable kernels $T(x_0,x_1,y)$ on $\R^9$ as follows: 
\be\begin{aligned}\lb{eq2.21}
(T_1 \oast T_2)(x_0, x_2, y) = \int_{\set R^6} T_1(x_0, x_1, y_1) T_2(x_1, x_2, y-y_1) \dd x_1 \dd y_1.
\end{aligned}\ee
Dually (i.e., on the Fourier side), $\oast$ takes the from
\EQ{\lb{comp}
&(\mc F^{-1}_{x_0} \mc F_{x_2, y} (T_1 \oast T_2))(\xi_0, \xi_2, \eta) \\
& = \int_{\set R^3} (\mc F^{-1}_{x_0} \mc F_{x_1, y} T_1)(\xi_0, \xi_1, \eta) (\mc F^{-1}_{x_1} \mc F_{x_2, y} T_2)(\xi_1, \xi_2, \eta) \dd \xi_1.
}
Thus, $\oast$ consists of convolution in the $y$ variable -- i.e., multiplication in the dual variable $\eta$ -- and composition of operators  relative to the other two. 
In the dual variables $\xi_0$, $\xi_1$, and $\xi_2$  composition of operators is preserved.  Note the order of the variables: $x_0$ is the ``input", $x_2$ the ``output" variable, whereas $y$ is the dual energy variable. 
\end{definition}

\medskip
We will study $\oast$ more systematically in Section~\ref{sec:FA}. For now, 
Lemma~\ref{lem:Tast} serves as an example of how we use $\oast$ to recursively generate all  $W_{n+}^\eps $, $n \geq 1$, starting from $W_{1+}^\eps $.  We use the dual formulation~\eqref{comp}
to define $\oast$ rather than the convolution~\eqref{eq2.21}. 

\begin{lemma}\lb{lem:Tast} 
Let $V$ be a Schwartz potential. 
For any $\eps>0$ we have for any $n,m\ge1$
$$
T_{m+}^\eps  \oast T_{n+}^\eps  = T_{(m+n)+}^\eps
$$
in the sense of \eqref{comp}. 
\end{lemma}
\begin{proof}
This follows from \eqref{2.16}. For example, 
\EQ{
\mc F^{-1}_{x_0} \mc F_{x_2, y} T_{2+}^\eps(\xi_0, \xi_2, \eta) &= \int_{\set R^3} \frac {\widehat V(\xi_2-\xi_1)}{|\xi_2 + \eta|^2 - |\eta|^2+i\eps} \cdot \frac {\widehat V(\xi_1-\xi_0)}{|\xi_1 + \eta|^2 - |\eta|^2+i\eps} \dd \xi_1   \\
& = \mc F^{-1}_{x_0} \mc F_{x_2, y} (T_{1+}^\eps\oast T_{1+}^\eps)(\xi_0, \xi_2, \eta)
}
both in the pointwise sense, as well as in the space of distributions. 
\end{proof}

The following lemma exhibits the relation with the resolvent operators. 

\begin{lemma}\lb{lem:FTe}
Let  $V$ be a Schwartz function, and $\eps>0$. Then 
\EQ{\lb{eq:fafb}
\mc F^{-1}_{x_0} \mc F_{x_1, y} T_{1+}^\eps (\xi_0, \xi_1, \eta) = \mc F^{-1}_a \mc F_b (R_0(|\eta|^2 - i\eps ) V)(\xi_0+\eta, \xi_1+\eta)
}
and, for any $n\ge1$, 
\EQ{\lb{eq:fafbn}
\mc F^{-1}_{x_0} \mc F_{x_1, y} T_{n+}^\eps (\xi_0, \xi_1, \eta) = \mc F^{-1}_a \mc F_b \big((R_0(|\eta|^2 - i\eps ) V)^n\big) (\xi_0+\eta, \xi_1+\eta).
}
Furthermore, 
\EQ{\lb{eq:t1'}
\mc F_y T_{1+}^\eps (x_0, x_1, \eta) = e^{ix_0\eta} \big(R_0(|\eta|^2 - i\eps ) V\big)(x_0, x_1) e^{-ix_1\eta}.
}
and
\be\begin{aligned}\lb{2.24}
\mc F^{-1}_{x_0} \mc F_{x_1, y} T_+^\eps (\xi_0, \xi_1, \eta) &= \mc F^{-1}_a \mc F_b (R_V(|\eta|^2 - i\eps ) V)(\xi_0+\eta, \xi_1+\eta).
\end{aligned}\ee
\end{lemma}
\begin{proof}
One has 
\EQ{\nonumber
\mc F^{-1}_a \mc F_b (R_0(|\eta|^2 - i\eps ) V)(\alpha,\beta) &=  \int_{\R^{6}} e^{i(a\cdot\alpha-b\cdot\beta)} R_0(|\eta|^2 - i\eps )(b-a) V(a)\, dadb\\
&= \int_{\R^{6}} e^{i(a\cdot(\alpha-\beta)-b\cdot\beta)} R_0(|\eta|^2 - i\eps )(b) V(a)\, dadb \\
&= \what{V}(\beta-\alpha) (|\beta|^{2}-|\eta|^{2}+i\eps)^{-1}
}
Plugging $\alpha=\xi_0+\eta$ and $\beta= \xi_1+\eta$  into the right-hand side yields \eqref{2.16}, which gives~\eqref{eq:fafb}.  From this, \eqref{eq:t1'} follows easily.   The representation of $T_{n+}^\eps$ follows using the composition $\oast$ from above. 
The other properties are left to the reader. 
\end{proof}

Lemma~\ref{lemma2.3}  immediately yields the following boundedness. As usual, $\B(X_1, X_2)$ are the bounded operators $X_1\to X_2$ for any Banach spaces $X_1$ and $X_2$. 

\begin{lemma}\lb{lem:0regBd}
Assuming that $V\in L^{\frac32,1}$ one has  $(\mc F_y T_{1+}^\eps )(\eta) \in \B(L^{\infty}_{x_0}, L^{\infty}_{x_1})$ and $(\mc F_y T_{1+}^\eps )(\eta) \in \B(L^1_{x_1}, L^1_{x_0})$ uniformly in $\eps>0$ and $\eta\in\R^3$.
Let $0$ energy be regular as in Definition~\ref{def:0res}. 
Then uniformly in $\eps>0$ and $\eta \in \R^3$, one has $(\mc F_y T_+^\eps )(\eta) \in \B(L^{\infty}_{x_0}, L^{\infty}_{x_1})$ and $(\mc F_y T_+^\eps )(\eta) \in \B(L^1_{x_1}, L^1_{x_0})$.  The respective operator norms are bounded by $C\|V\|_{L^{\frac32,1}}$ with some absolute constant $C$. 
\end{lemma}
\begin{proof}
The statements about $T_{1+}^\eps$ are easily obtained from \eqref{eq:Ho1}. The second statement concerning $L^1$ boundedness follows from the first by duality (note the reversal of the order of the variables). 
As far as $T_+^\eps$ is concerned, the first statement is Lemma~\ref{lemma2.3}, whereas the second follows by duality. 
\end{proof}

In the sequel, we shall employ the following form of the operators $W_{n+}^\eps $ and $W_+^\eps $, introduced by Yajima in \cite{yajima0}.

\begin{lemma}\lb{lemma2.1} 
Let $V$ be Schwartz. Then for any Schwartz functions  $f,g$ and 
for $\eps >0$, $n \geq 1$,
\be\begin{aligned}\lb{2.19}
\langle W_{n+}^\eps  f, g \rangle &= (-1)^n \int_{\set R^9} \mc F_{x_{0}}^{-1}T_{n+}^\eps (0, x, y) f(x-y) \ov g(x) \dd y \dd x
\end{aligned}\ee
and
\be\begin{aligned}\lb{2.20}
\langle W_+^\eps  f, g \rangle &= \langle f, g \rangle - \int_{\set R^9} \mc F_{x_{0}}^{-1} T_+^\eps (0, x, y) f(x-y) \ov g(x)  \dd y \dd x.
\end{aligned}\ee
These integrals are to be understood as distributional duality pairings. 
All our conclusions concerning $T_+^\eps $ apply equally to $T_-^\eps $.
\end{lemma}
\begin{proof}
By Plancherel's identity
$$\begin{aligned}
\langle W_{1+}^\eps  f, g\rangle &= \frac i {(2\pi)^3} \int_0^{\infty} \int_{\R^6} e^{it|\eta_1|^2-\eps  t} \widehat V(\eta_1 - \eta_0) e^{-it|\eta_0|^2} \widehat f(\eta_0) \ov{\widehat g}(\eta_1) \dd \eta_1 \dd \eta_0 \dd t \\
&= - \frac 1 {(2\pi)^3} \int_{\R^6} \frac {\widehat V(\eta_1 - \eta_0)} {|\eta_1|^2 - |\eta_0|^2 + i \eps } \widehat f(\eta_0) \ov {\widehat g}(\eta_1) \dd \eta_1 \dd \eta_0.
\end{aligned}
$$
Setting $\eta_0 = \eta$,  $\eta_1 - \eta_0 = \xi$, we obtain  
\be\lb{eqn2.6}
\langle W_{1+}^\eps  f, g\rangle = - \frac 1 {(2\pi)^3} \int_{\R^6} \frac {\widehat V(\xi)}{|\eta+\xi|^2- |\eta|^2 + i \eps } \widehat f(\eta) \ov{\widehat g}(\eta+\xi) \dd \eta \dd \xi.
\ee
More generally, when $n \geq 1$,  
$$\begin{aligned}
\langle W_{n+}^\eps  f, g\rangle &= \frac {i^n}{(2\pi)^3} \int_{0=t_0\leq t_1\leq \ldots\leq t_n} \prod_{\ell = 1}^n \Big(e^{i(t_{\ell} - t_{\ell-1})|\eta_{\ell}|^2-(t_{\ell}-t_{\ell-1})\eps } \widehat V(\eta_{\ell}-\eta_{\ell-1}) \Big) \\
&e^{-it_n|\eta_0|^2} \widehat f(\eta_0) \ov{\widehat g}(\eta_n) \dd \eta_0 \ldots \dd \eta_n \dd t_1 \ldots \dd t_n,
\end{aligned}$$
with $t_0 = 0$. Integrating in $t_1, \ldots, t_n$ gives
$$\begin{aligned}
\langle W_{n+}^\eps  f, g\rangle &= \frac {(-1)^n}{(2\pi)^3} \int_{\R^{3(n+1)}} \frac {\prod_{\ell=1}^n \widehat V(\eta_{\ell}-\eta_{\ell-1})} {\prod_{\ell=1}^n (|\eta_{\ell}|^2 - |\eta_0|^2 + i \eps )} \widehat f(\eta_0) \ov {\widehat g}(\eta_n) \dd \eta_0 \ldots \dd \eta_n.
\end{aligned}$$
Renaming  $\eta_0 = \eta$,  $\eta_{\ell} -\eta_0 = \xi_{\ell}$ leads to 
\be\begin{aligned}\lb{eqn2.7}
\langle W_{n+}^\eps  f, g\rangle = \frac {(-1)^n}{(2\pi)^3} \int_{\R^{3(n+1)}} \frac{\prod_{\ell=1}^n \widehat V(\xi_{\ell} - \xi_{\ell-1}) \dd \xi_1 \ldots \dd \xi_{n-1}}{\prod_{\ell=1}^n (|\eta+\xi_{\ell}|^2- |\eta|^2 + i \eps )} \widehat f(\eta) \ov{\widehat g}(\eta+\xi_n) \dd \eta \dd \xi_n.
\end{aligned}\ee
with $\xi_0 = 0$. 
Then
\be\lb{dual_formula}\begin{aligned}
\langle W_{n+}^\eps  f, g \rangle &= \frac{(-1)^n}{(2\pi)^3} \int_{\set R^6} \mc F^{-1}_{x_0} \mc F_{x_n, y} T_{n+}^\eps (0, \xi_n, \eta) \widehat f(\eta) \ov {\widehat g}(\eta + \xi_n) \dd \eta \dd \xi_n \\
&= (-1)^n \int_{\set R^9} \mc F_{x_{0}}^{-1 }T_{n+}^\eps (0, x, y) f(x-y) \ov g(x) \dd y \dd x.
\end{aligned}\ee
As far as the wave operators are concerned, we have 
$$\begin{aligned}
\langle W_+^\eps  f, g\rangle &=  \langle f, g \rangle+ \frac i {(2\pi)^3} \int_0^{\infty} \int_{\R^6} \mc F^{-1}_a \mc F_b \big(e^{itH-t\eps }V\big)(\eta_0, \eta_1) e^{-it|\eta_0|^2} \widehat f(\eta_0) \ov{\widehat g}(\eta_1) \dd \eta_1 \dd \eta_0 \dd t \\
&= \langle f, g \rangle - \frac 1 {(2\pi)^3} \int_{\R^6} \mc F^{-1}_a \mc F_b \big(R_V(|\eta_0|^2 - i\eps ) V\big)(\eta_0, \eta_1) \widehat f(\eta_0) \ov{\widehat g}(\eta_1) \dd \eta_1 \dd \eta_0 \\
&= \langle f, g \rangle - \frac 1 {(2\pi)^3} \int_{\R^6} \mc F^{-1}_a \mc F_b (R_V(|\eta|^2 - i\eps ) V)(\eta, \eta+\xi) \widehat f(\eta) \ov {\widehat g}(\eta + \xi) \dd \eta \dd \xi.
\end{aligned}
$$
Then 
\be\begin{aligned}\lb{eqn2.29}
\langle W_+^\eps  f, g \rangle &= \langle f, g \rangle - \frac 1 {(2\pi)^3} \int_{\set R^6} \mc F^{-1}_{x_0} \mc F_{x_1, y} T_+^\eps (0, \xi_1, \eta) \widehat f(\eta) \ov {\widehat g}(\eta + \xi_1) \dd \eta \dd \xi_1 \\
&= \langle f, g \rangle - \int_{\set R^9} \mc F_{x_{0}}^{-1}T_+^\eps (0, x, y) f(x-y) \ov g(x) \dd y \dd x
\end{aligned}\ee
as desired. 
\end{proof}


\section{The first order term in the Born series}
\label{sec:Born1}

We now set out to analyze the operator
\[
W_{1+}^\eps  f  = i\int_0^\infty e^{-it\Delta} V e^{it\Delta-t\eps }f\, dt
\]
introduced in the previous section. From \eqref{2.19}, for Schwartz functions $V$ and $f$, 
\EQ{\label{eq:KR}
W_{1+}^\eps  f(x) &=  \int_{\R^3} K_{1+}^\eps (x,x-y) f(y)\, dy \\
K_{1+}^\eps (x,z) &= -\lim_{R\to\infty} \int_{\R^6}  \frac{e^{ix\cdot \xi}\, \what{V}(\xi)\,e^{iz\cdot \eta}}{|\xi+\eta|^2-|\eta|^2+i\eps } e^{-\frac{|\eta|^2}{2R^2}}\,d\xi d\eta
}
The Gaussian was introduced to ensure convergence of the $\eta$ integral. In this section, we will show the existence of the limit in~\eqref{eq:KR} and find the kernel. 

For future reference, we remark that by (\ref{eqn2.6}) the kernel associated to $W_{1+}$ is $-T_{1+}$, i.e., 
\EQ{ \lb{eq:K1+T}
(\mc F_{x,y}K_{1+}^\eps) (\xi,\eta) &= -(\mc F^{-1}_{x_0} \mc F_{x, y} T_{1+}^\eps )(0, \xi, \eta) \\
& =  -\frac {\widehat V(\xi)}{|\eta+\xi|^2 - |\eta|^2 + i \eps }\\
K_{1+}^\eps (x,y) &= - \int_{\R^{3}} T_{1+}^{\eps}(x_{0},x,y)\, dx_{0}
}
where the final equality is formal. 
Integrating in $x_0$, as in  (\ref{2.19}) and (\ref{2.20}), corresponds to setting $\xi_0 = 0$ in (\ref{2.16}) and (\ref{2.17}).  

\begin{lemma}
For any $\eps >0$ and all $a, x\in\R$ one has 
\EQ{\label{eq:PV}
\text{P.V.} \int_{-\infty}^\infty \frac{e^{ix\xi}}{\xi+a+ i\eps }\, d\xi = -2\pi i \one_{[x<0]} e^{\eps  x}  e^{-ixa}
}
The left-hand side agrees with the inverse distributional Fourier transform of $(\xi+a+i\eps )^{-1}$. 
\end{lemma}
\begin{proof}
This is a standard residue calculation. 
\end{proof}

The denominator in \eqref{eq:KR} is $|\xi+\eta|^2-|\eta|^2+i\eps  = |\xi|^2+2\xi\cdot\eta +i\eps $.  Up to a rotation this leads to an integral of the type~\eqref{eq:PV}. 

\begin{lemma} \label{lem:G2}
For any $\eps >0$, $R>0$, and $\xi\in\R^3$, $\xi\ne0$,  one has 
\EQ{\label{eq:G2} 
&\int_{\R^3} \frac{e^{iz\cdot \eta} e^{-\frac{|\eta|^2}{2R^2}}}{|\xi|^2+2\xi\cdot\eta +i\eps } \, d\eta  =\const \frac{R^2}{|\xi|} e^{-\frac{R^2}{2} |P_\xi^\perp z |^2} 
 \int_{-\infty}^0 e^{\frac{\eps u}{2|\xi|}}  e^{-i \frac{u}{2}|\xi|} \: R e^{-\frac{R^2}{2} (z\cdot \hat{\xi} -u)^2}\, du
}
with $\hat{\xi}=\frac{\xi}{|\xi|}$ and $|P_\xi^\perp z|^2= |z|^2-(\hat \xi \cdot z)^2$. 
\end{lemma}
\begin{proof}
We first treat the case $\xi=(|\xi|,0,0)$. Then
the left-hand side of \eqref{eq:G2}  equals 
\EQ{
& \frac{1}{8|\xi|^3} \int_{\R^3} \frac{e^{i\frac{z\cdot \eta}{2|\xi|}}\: e^{-\frac{|\eta|^2}{8R^2|\xi|^2}} }{\hat\xi\cdot \eta + |\xi|^2 +i\eps}\, d\eta \\
&  =  \frac{1}{8|\xi|^3} \int_{\R} \frac{e^{i\frac{z_1 \eta_1}{2|\xi|}}\: e^{-\frac{\eta_1^2}{8R^2|\xi|^2}} }{  \eta_1 + |\xi|^2 +i\eps}\, d\eta_1 \int_{\R^2} e^{i \frac{z'\cdot \eta'}{2|\xi|} } \:  e^{-\frac{ |\eta'|^2 }{8R^2 |\xi|^2 }} \, d\eta'  \\
&= \const R^2 e^{-\frac{R^2}{2}|z'|^2} R \int_{-\infty}^0 e^{\eps v} e^{-iv|\xi|^2} e^{-2R^2 |\xi|^2 \big( \frac{z_1}{2|\xi|} - v\big)^2}\, dv 
}
where we used \eqref{eq:PV} in the final equality. Substituting $v=\frac{u}{2|\xi|}$ in the integral on the right-hand side shows that the previous line equals 
\[
\const \frac{R^2}{|\xi|} e^{-\frac{R^2}{2} |z'|^2} \int_{-\infty}^0 e^{\frac{\eps u}{2|\xi|}}  e^{-i \frac{u}{2}|\xi|} \: R e^{-\frac{R^2}{2} (z_1-u)^2}\, du
\]
as desired. The general case now follows by rotating the coordinate frame. 
\end{proof}

In the limit $R\to\infty$ the right-hand side of \eqref{eq:G2} converges in the sense of distributions to 
\[
\const |\xi|^{-1} \delta_0( P_\xi^\perp z) \one_{[z\cdot\xi<0]} \; e^{\frac{\eps z\cdot \hat{\xi}}{2|\xi|} } e^{-i{z\cdot \xi}/2}  
\]
We can now compute the kernel in \eqref{eq:KR}. 

\begin{lemma}\label{lem:Kexp}
Assume $V$ is a Schwartz potential. 
For any  $\eps\ge0$ and $x,z\in \R^3$, $z\ne0$, one has 
\EQ{\label{eq:K1V}
K^\eps _{1+}(x,z) = \const |z|^{-2}  \int_0^\infty e^{-is\hat{z}\cdot(x-z/2)} \widehat{V}(-s\hat{z}) e^{-\eps \frac{|z|}{2s}}\: s\,ds
}
where $\hat{z}=z/|z|$. 
\end{lemma}
\begin{proof}
By \eqref{eq:KR},
\EQ{
K^\eps_{1+}(x,z) &= \lim_{R\to\infty} \int_{\R^6}  \frac{e^{ix\cdot \xi} \;\what{V}(\xi)e^{iz\cdot \eta}}{|\xi+\eta|^2-|\eta|^2+i\eps } e^{-\frac{|\eta|^2}{2R^2}}\,d\xi d\eta \\
 &=  \const \lim_{R\to\infty} \int_{\R^3}  e^{ix\cdot \xi}  \frac{\what{V}(\xi) }{|\xi|}   R^2\exp\big( - \frac{R^2}{2}  |P_\xi^\perp z|^2\big) \\ %
 &\qquad\qquad\cdot \int_{-\infty}^0  e^{\frac{\eps u}{2|\xi|}}  e^{-i \frac{u}{2}|\xi|} \: R e^{-\frac{R^2}{2} (z\cdot \hat{\xi} -u)^2}\, du    \,  d\xi
}
We set 
\[
J(\xi;\eps,R,z) =  \int_{-\infty}^0  e^{\frac{\eps u}{2|\xi|}}  e^{-i \frac{u}{2}|\xi|} \: R e^{-\frac{R^2}{2} (z\cdot \hat{\xi} -u)^2}\, du
\]
Then 
\EQ{
0\le J(\xi;\eps,R,z)  &\le \int_{-\infty}^\infty e^{-u^2/2}\, du = \sqrt{2\pi} \\
J (\xi;\eps,R,z)  &\to  \const \one_{[z\cdot\xi<0]} e^{-iz\cdot\xi/2} e^{\eps \frac{z\cdot\hat{\xi}}{2|\xi|}} \text{ \ \ as\ \ }R\to\infty
}
whence, writing $\xi=-s\omega$, $|\omega|=1$, 
\EQ{\label{eq:Keps1+}
K^\eps_{1+}(x,z) &= \const \lim_{R\to\infty} \int_0^\infty \int_{\Sph^2 } e^{-isx\cdot \omega} \; \what{V}(-s\omega) \\ 
&\qquad R^2 e^{-\frac12R^2 |z|^2 \sin^2(\angle(z,\omega))}
J(-s\omega;\eps,R,z) s\, dsd\omega \\
&= \const |z|^{-2}\int_0^\infty e^{-isx\cdot \hat{z}} \; \what{V}(-s\hat{z})e^{is|z|/2} e^{-\eps \frac{|z|}{2s}}\: s\,ds   
}
which is \eqref{eq:K1V}.
\end{proof}

Following Yajima, we express the kernel in Lemma~\ref{lem:Kexp} in terms of the function $L$ which we now define. 

\begin{definition}\label{def:L}
Let $V$ be a Schwartz potential. Then for any $\omega\in \Sph^2 $, $r\in\R$
\EQ{\label{eq:L}
L(r,\omega) := \int_0^\infty \what V(-s\omega)e^{i\frac{rs}{2}} \,s\, ds
}
\end{definition}

In particular, from \eqref{eq:K1V}, 
\EQ{\label{eq:KL}
K_{1+} (x,z) &=  \const |z|^{-2}   L(|z|-2x\cdot\hat z, \hat z)
}
The kernels $K_{1+}^\eps (x,z)$ are of  the form, see \eqref{eq:Keps1+},  
 \EQ{\label{eq:KLeps}
K_{1+}^\eps (x,z) &=  \const     L_\eps(|z|-2x\cdot\hat z, z)
}
where 
\EQ{\label{eq:Leps}
L_\eps (r, \omega) := \int_0^\infty \what V(-s\omega)e^{i\frac{rs}{2}|\omega|} e^{-  \frac{\eps}{2s}}\: s\, ds
}
In the previous line $\omega$ need not be a unit vector. 

\begin{cor}
\label{cor:K1+}
Assume $V$ is Schwartz. Let $S_\omega := x-2(\omega\cdot x) \omega$ be the reflection about the plane $\omega^\perp$. Then for all Schwartz functions $f$ one has 
\EQ{\label{eq:g1}
(W_{1+}f)(x) = \int_{\Sph^2 }\int_{\R^3} g_1(x,dy,\omega) f(S_\omega x-y)\,  \sigma(d\omega)
}
where  for fixed $x\in\R^3$, $\omega\in \Sph^2 $ the expression $g_1(x,\cdot,\omega)$ is a measure satisfying 
\EQ{\lb{eq:g1meas}
 \int_{\Sph^2 }  \| g_1(x,dy,\omega)\|_{\mes_{y} L^\infty_x}  \, d\omega \le \int_{\Sph^2 }\int_{\R} |L(r,\omega)|\, drd\omega 
}
where $\|\cdot \|_{\mes}$ refers to the total variation norm of Borel measures. 
\end{cor}
\begin{proof}
Eq.~\eqref{eq:KR} and \eqref{eq:K1V} imply that
\EQ{
(W_{1+}f)(x) &= \int_0^\infty \int_{\Sph^2 } L(r-2\omega\cdot x,\omega) f(x-r\omega)\, drd\omega \\
&= \int_{\Sph^2 } \int_{\R} \one_{[r>-2\omega\cdot x]} L(r,\omega) f(x-2(\omega\cdot x) \omega -r\omega)\, drd\omega
}
Set
\EQ{\lb{eq:g1def}
g_{1}(x,dy,\omega):= \one_{[(y+2x)\cdot\omega>0]} L(y\cdot\omega,\omega) \, \calH^1_{\ell_\omega}(dy) 
}
where $\ell_\omega=\{r\omega\:|\: r\in\R\}$ and $\calH^1_{\ell_\omega}$ is the $1$--dimensional Hausdorff measure on the line $\ell_\omega$. Then \eqref{eq:g1} holds and 
\EQ{\lb{eq:g1sup}
\| g_1(x,dy,\omega)\|_{L^\infty_x} = | L(y\cdot\omega,\omega)| \, \calH^1_{\ell_\omega}(dy) 
}
which implies \eqref{eq:g1meas}. 
\end{proof}


\section{Estimating the function $L(r,\omega)$ and the kernel $K_{1+}$}

Next, we estimate various norms of $L(r,\omega)$ such as the one in \eqref{eq:g1meas}, in terms of norms of $V$ in the $B$-spaces from Definition~\ref{def:Bspace}. 

\begin{prop}\lb{lemma2.9} 
Let $L$ be as in Definition~\ref{def:L}, and $V$ be a Schwartz function. Then, with $r\in\R$ and $\omega\in \Sph^2 $, 
\EQ{\label{eq:L2V}
\|L(r,\omega)\|_{L^2_{r, \omega}} \les \|V\|_{L^2}
}
and
\be\lb{est_b}
\begin{aligned}
\|L(r,\omega)\|_{L^1_{r, \omega}} \les \sum_{k \in \Z} 2^{k/2} \|\one_{[2^k, 2^{k+1}]}(|r|) L(r,\omega)\|_{L^2_{r, \omega}} \les \|V\|_{\dot B^{\frac12}}\les \|V\|_{B^{\frac12}}.
\end{aligned}\ee
Moreover,
if   $0 < \alpha < 1$, then
\be\lb{2.81}
\sum_{k \in \Z} 2^{\alpha k} \|\one_{[2^k, 2^{k+1}]}(|r|) L(r,\omega)\|_{L^2_{r, \omega}} \les \|V\|_{\dot B^\alpha}.
\ee
\end{prop}
\begin{proof} By Plancherel's identity,
$$
\int_\R |L(r, \omega)|^2 \dd t \les \int |\widehat V(s\omega)|^2 s^2  \dd s.
$$
Integrating over $\omega \in \Sph^2 $  yields
\be\lb{2.84}
\|L\|_{L^2_{r, \omega}}^2 = \int_{\Sph^2 } \int_\R |L(r, \omega)|^2 \dd r \dd \omega \les \|\widehat V\|_{L^2}^2 = C \|V\|_{L^2}^2.
\ee
Integrating  by parts one obtains
$$
\frac {ir} 2 L(r, \omega) = -\int_0^{\infty} \partial_s (\widehat V(-s\omega) s) e^{irs/2} \dd s,
$$
whence again by Plancherel, 
\be\lb{2.85}
\|r L\|_{L^2_{r, \omega}}^2 = \int_{\Sph^2 } \int_\R |r L(r, \omega)|^2 \dd r \dd \omega \les \|\dl \widehat V\|_2^2 + \||\xi|^{-1} \widehat V\|_{L^2_\xi}^2 \les \|V\|_{|x|^{-1} L^2}^2.
\ee
When $V \in \dot B^{\frac12}$ we now prove (\ref{est_b}) by applying the real interpolation method, see Section~\ref{sec:Lpq}. Begin by partitioning $L$ into dyadic pieces
$$
L^{j}(r, \omega) = L(r, \omega) (\chi(2^{-j-1} |r|) - \chi(2^{-j+1} |r|)), \quad\calL :=\{ L^{j}\}_{j\in\Z}
$$
where $\chi$ is  a cutoff function such that $\chi(s)=1$ when $s \leq 1$ and $\chi(s)=0$ when $s\geq 2$. 
When $V \in L^2$, we can then rewrite (\ref{2.84}) in the form (using the vector spaces from \eqref{eq:ellsq}) 
$$
\calL (r, \omega) \in \dot\ell^0_2(L^2_{r, \omega}),\ \|\calL(r, \omega)\|_{\dot\ell^0_2(L^2_{r, \omega})} \les \|V\|_2.
$$
Likewise, (\ref{2.85}) becomes
$$
\calL(r, \omega) \in   \dot\ell^{1}_2(L^2_{r, \omega}),\ \|\calL(r, \omega)\|_{  \dot\ell^1_2(L^2_{r, \omega})} \les \|V\|_{|x|^{-1} L^2}  
$$
By real interpolation, see Lemma~\ref{lem:Balpha interpol} and \eqref{eq:vec int}, 
$$
\|\calL(r, \omega)\|_{\dot \ell^{\alpha}_1(L^2_{r, \omega})} \les \|V\|_{\dot B^{\alpha}},
$$
for any $0<\alpha<1$. In particular, 
$$
\|L\|_{L^1_{r, \omega}} \les \|\calL(r,\omega)\|_{\dot \ell^0_1(L^1_{r, \omega})} \les\|\calL(r,\omega)\|_{  \dot\ell^{\frac12}_2(L^2_{r, \omega})} \les \|V\|_{\dot B^{\frac12}}.
$$
This establishes \eqref{est_b}, \eqref{2.81} in the range $0<\alpha<1$. 
\end{proof}

As an immediate corollary we obtain via \eqref{eq:g1def}, \eqref{eq:g1sup} that 
\EQ{\lb{eq:g1measV}
 \int_{\Sph^2 }\int_{\R^3} \| g_1(x,dy,\omega)\|_{L^\infty_x}  \, d\omega \les \|V\|_{B^{\frac12}}
}
from which we deduce via Corollary~\ref{cor:K1+} that 
\[
\| W_{1+}f\|_p \le C \|V\|_{B^{\frac12}} \|f\|_p
\]
In order to bound other terms $W_{n+}$ as well as the full wave operator we will rely on a certain function algebra framework that is presented in Section~\ref{sec:FA}. 
This function algebra formalism will need to respect the composition law of Definition~\ref{def:comp}. 
To motivate the definitions of the function spaces in the following section,  we now establish some estimates on the kernel $K_{1+}^\eps$.  
Recall from \eqref{eq:K1+T} that 
\EQ{
K_{1+}^\eps (x,y) &=   -\mc F_{x_{0}}^{-1} T_{1+}^\eps (0,x,y)
}

\begin{lemma}
\label{lem:K1}
Let $V$ be a Schwartz function. Then for any $0\le\sigma$, and uniformly in $\eps>0$, 
\begin{align}
\| K_{1+}^\eps (x,y)\|_{L^\infty_x L^1_y} & \les \| V\|_{B^{\frac12}} \lb{eq:Kest1}  \\
\| v(x) K_{1+}^\eps (x,y)\|_{ L^1_y B^{\sigma}_x} & \les \| v\|_{B^{\frac12+\sigma}} \| V\|_{B^{\frac12}}     \lb{eq:Kest2} 
\end{align}
for any $v\in B^{\frac12+\sigma}$.  With $f$ a Schwartz function, define a kernel 
\EQ{
\wtil K_{1+}^\eps (x,y) &=  \int_{\R^3} f(x_0)   T_{1+}^\eps (x_0,x,y)\, dx_0
}
with the integral being understood as distributional duality pairing. 
Then  uniformly in $\eps>0$,
\begin{align}
\| \wtil K_{1+}^\eps (x,y)\|_{L^\infty_x L^1_y} & \les \| fV\|_{B^{\frac12}} \lb{eq:Kest1f}  \\
\| v(x) \wtil K_{1+}^\eps (x,y)\|_{ L^1_y B^{\sigma}_x} & \les \| v\|_{B^{\frac12+\sigma}} \| fV\|_{B^{\frac12}}     \lb{eq:Kest2f} 
\end{align}
for any $v\in B^{\frac12+\sigma}$. 
\end{lemma}
\begin{proof} 
It suffices to give the proof for the limit $\eps=0$ since for $\eps>0$ an additional exponential decay factor arises, cf.~\eqref{eq:L} and~\eqref{eq:Leps}.
So all estimates below cover the case $\eps>0$ as well.  
From \eqref{eq:KL} one has 
\EQ{
\| K_{1+}  (x,y)\|_{L^\infty_x L^1_y} &= \const \int_{-\infty}^\infty \int_{\Sph^2} |L(r,\omega)|\, drd\omega \les \|V\|_{B^\frac12} 
}
as we already noted above, cf.~\eqref{eq:g1measV}.   For the second estimate \eqref{eq:Kest2} we proceed as follows: 
\begin{align}
  \| v(x) K_{1+}  (x,y)\|_{ L^1_y B^{\sigma}_x} & = \const \int_{\R^3} \Big \| v(x) |u|^{-2} L(|u|-2\hat u \cdot x, \hat u)  \Big \|_{   B^{\sigma}_x}  \, du \nn \\ 
&= \const \int_{\Sph^2} \int_0^\infty  \| v(x)   L(r-2\omega \cdot x, \omega)  \Big \|_{   B^{\sigma}_x}  \, dr d\omega  \label{eq:vK1} \\
&\les  \int_{\Sph^2} \int_0^\infty   \Big \| \one_{[|x|\les 1]} v(x)   L(r-2\omega \cdot x, \omega)  \Big  \|_{ L^2_x} \, dr d\omega  + \nn \\
& \quad +\sum_{\ell=0}^\infty 2^{\sigma\ell} \int_{\Sph^2} \int_0^\infty   \Big \| \one_{[|x|\simeq 2^\ell]} v(x)   L(r-2\omega \cdot x, \omega)  \Big  \|_{ L^2_x} \, dr d\omega   \nn
\end{align}
The term involving $\one_{[|x|\les 1]}$ is estimated as follows: 
\EQ{\nonumber
& \int_{\Sph^2} \int_0^\infty   \Big \| \one_{[|x|\les 1]} v(x)   L(r-2\omega \cdot x, \omega)  \Big  \|_{ L^2_x} \, dr d\omega  \\
& \les  \Big( \int_{\Sph^2} \int_0^1  \Big \| \one_{[|x|\les 1]} v(x)   L(r-2\omega \cdot x, \omega)  \Big  \|_{ L^2_x}^2 \, dr d\omega\Big)^{\frac12} + \\
&\qquad + \sum_{k=0}^\infty 2^{\frac{k}{2}} \Big( \int_{\Sph^2} \int_{2^k\le r\le 2^{k+1}}  \Big \| \one_{[|x|\les 1]} v(x)   L(r-2\omega \cdot x, \omega)  \Big  \|_{ L^2_x}^2 \, dr d\omega\Big)^{\frac12}  \\
&\les \big \| \one_{[|x|\les 1]} v(x)  \big\|_{2} \sum_{k=0}^\infty 2^{\frac{k}{2}}  \Big( \int_{\Sph^2} \int_{2^k-1\le |r| \le 2^{k+2}} |   L(r , \omega)  |^2 \, dr d\omega\Big)^{\frac12} \les  \|v\|_2 \|V\|_{B^{\frac12}}
}
by Proposition~\ref{lemma2.9}.  Next, 
\EQ{\nn
& \sum_{\ell=0}^\infty 2^{\sigma\ell} \int_{\Sph^2} \int_0^\infty   \Big \| \one_{[|x|\simeq 2^\ell]} v(x)   L(r-2\omega \cdot x, \omega)  \Big  \|_{ L^2_x} \, dr d\omega   \\
& \les \sum_{\ell=0}^\infty 2^{\sigma\ell} \sum_{k\ge \ell + 10} 2^{\frac{k}{2}}\Big(  \int_{\Sph^2} \int_{[r\simeq 2^k]} \Big \| \one_{[|x|\simeq 2^\ell]} v(x)   L(r-2\omega \cdot x, \omega)  \Big  \|_{ L^2_x}^2 \, dr d\omega   \Big)^{\frac12}   \\
&\quad + \sum_{\ell=0}^\infty   2^{\ell(\frac12+\sigma)}\Big(  \int_{\Sph^2} \int_{[0 <r\les 2^\ell]} \Big \| \one_{[|x|\simeq 2^\ell]} v(x)   L(r-2\omega \cdot x, \omega)  \Big  \|_{ L^2_x}^2 \, dr d\omega   \Big)^{\frac12}    
}
In the first sum on the right-hand side $r$ dominates, but for the second it does not. Hence, we continue to bound these terms as follows: 
\EQ{
&\les   \sum_{\ell=0}^\infty 2^{\sigma\ell} \| \one_{[|x|\simeq 2^\ell]} v(x)\|_2 \sum_{k\ge \ell + 10} 2^{\frac{k}{2}}\Big(  \int_{\Sph^2} \int_{[r\simeq 2^k]} \big |    L(r, \omega)  |^2 \, dr d\omega   \Big)^{\frac12}   \\
&\qquad + \sum_{\ell=0}^\infty   2^{\ell(\frac12+\sigma)}   \| \one_{[|x|\simeq 2^\ell]} v(x)   \|_2 \Big(  \int_{\Sph^2} \int_{[|r| \les 2^\ell]}    | L(r, \omega)  |^2 \, dr d\omega   \Big)^{\frac12}  \\
&\les \|v\|_{B^{\sigma+\frac12}} \|V\|_{B^{\frac12}}
}
as claimed.  Finally, we have the relation
\EQ{\label{eq:fV1}
\mc F_{x,y}\wtil K_{1+}^\eps (\xi_1,\eta) &=  \int_{\R^3} \hat{f}(\xi_0) \mc F^{-1}_{x_0} \mc F_{x,y}T_{1+}^\eps (\xi_0,\xi_1,\eta)\, d\xi_0 \\
&= \int_{\R^3} \frac{ \hat{f}(\xi_0) \what{V}(\xi_1-\xi_0)}{|\xi_1+\eta|^2-|\eta|^2+i\eps}\, d\xi_0 = \frac{ \what{fV}(\xi_1)}{|\xi_1+\eta|^2-|\eta|^2+i\eps}
}
In view of \eqref{2.16}, this corresponds to  the kernel $K_{1+}^\eps$ associated with the potential $fV$. So the previous estimates yield \eqref{eq:Kest1f}, \eqref{eq:Kest2f}. 
\end{proof}

Note that the proof of \eqref{eq:Kest2f} suffers a loss of a half power in the sense that $\|v\|_{B^1}$ appears on the right-hand side instead of $\|v\|_{B^\frac12}$.  
However, since an estimate of the form 
$
\| K_{1+}^\eps (x,y)\|_{L^1_yL^\infty_x }  \les \| V\|_{B^{\frac12}}
$
 is false, removing such a loss in the context of the $L^{2}$-based theory seems  delicate.  
  
 \smallskip
 
 In Section~\ref{sec:Wiener} we will make use of the following technical variant of \eqref{eq:Kest2}. While we only need the case $\gamma_1=\gamma_2$, we choose this more
 general formulation to illustrate the distribution of the different weights. 
 
 \begin{lemma}
\label{lem:K1weighted}
Let $V$ be a Schwartz function, and $0\le\gamma_1<\frac12$, $\gamma_2\ge0$, and set $\sigma=\frac12+\gamma_2$. Then uniformly in $\eps>0$, 
\begin{align}
\int_{\R^{3}} \la y\ra^{\gamma_1}\| v(x) K_{1+}^\eps (x,y)\|_{B^{\sigma}_x} \, dy & \les \| v\|_{B^{\frac12+\sigma+\gamma_1 }} \| V\|_{B^{\frac12+\gamma_1}}     \lb{eq:Kest2weighted} 
\end{align}
for any $v\in B^{\frac12+\sigma+\gamma_1 }$.  
\end{lemma}
\begin{proof}
The proof is a variant of  the one for the previous lemma. First, 
\begin{align*}
& \int_{\R^{3}} \la y\ra^{\gamma_1}  \| v(x) K_{1+}  (x,y)\|_{B^{\frac12+\gamma_2}_x} \, dy \\
 & = \const \int_{\R^3} \la u\ra^{\gamma_1} \Big \| v(x) |u|^{-2} L(|u|-2\hat u \cdot x, \hat u)  \Big \|_{   B^{\frac12+\gamma_2}_x}  \, du \nn \\ 
&= \const \int_{\Sph^2} \int_0^\infty \la r\ra^{\gamma_1}  \| v(x)   L(r-2\omega \cdot x, \omega)  \Big \|_{   B^{\frac12+\gamma_2}_x}  \, dr d\omega   \\
&\les  \int_{\Sph^2} \int_0^\infty   \Big \| \one_{[|x|\les 1]} v(x)   L(r-2\omega \cdot x, \omega)  \Big  \|_{ L^2_x} \la r\ra^{\gamma_1}\, dr d\omega  + \nn \\
& \quad +\sum_{\ell=0}^\infty 2^{\ell(\frac{1}{2}+\gamma_2)} \int_{\Sph^2} \int_0^\infty   \Big \| \one_{[|x|\simeq 2^\ell]} v(x)   L(r-2\omega \cdot x, \omega) 
 \Big  \|_{ L^2_x}\la r\ra^{\gamma_1} \, dr d\omega   \nn
\end{align*}
The term involving $\one_{[|x|\les 1]}$ is estimated as follows: 
\EQ{\nonumber
& \int_{\Sph^2} \int_0^\infty   \Big \| \one_{[|x|\les 1]} v(x)   L(r-2\omega \cdot x, \omega)  \Big  \|_{ L^2_x} \la r\ra^{\gamma_1}\, dr d\omega  \\
& \les    \Big( \int_{\Sph^2} \int_0^1  \Big \| \one_{[|x|\les 1]} v(x)   L(r-2\omega \cdot x, \omega)  \Big  \|_{ L^2_x}^2 \la r\ra^{\gamma_1}\, dr d\omega\Big)^{\frac12} + \\
&\qquad + \sum_{k=0}^\infty 2^{{k}(\frac12+\gamma_1)} \Big( \int_{\Sph^2} \int_{2^k\le r\le 2^{k+1}}  \Big \| \one_{[|x|\les 1]} v(x)   L(r-2\omega \cdot x, \omega)  \Big  \|_{ L^2_x}^2 \, dr d\omega\Big)^{\frac12}  \\
&\les \big \| \one_{[|x|\les 1]} v(x)  \big\|_{2} \sum_{k=0}^\infty 2^{k(\frac12+\gamma_1)}  \Big( \int_{\Sph^2} \int_{2^k-1\le |r| \le 2^{k+2}}  
|   L(r , \omega) |^2 \, dr d\omega\Big)^{\frac12} \les  \|v\|_2 \|V\|_{B^{\frac12+\gamma_1}}
}
by Proposition~\ref{lemma2.9}.  Next, 
\EQ{\nn 
& \sum_{\ell=0}^\infty 2^{\ell(\frac{1}{2}+\gamma_2)} \int_{\Sph^2} \int_0^\infty   \Big \| \one_{[|x|\simeq 2^\ell]} v(x)   L(r-2\omega \cdot x, \omega)  \Big  \|_{ L^2_x}\la r\ra^{\gamma_1} \, dr d\omega   \\
& \les \sum_{\ell=0}^\infty 2^{\ell(\frac{1}{2}+\gamma_2)} \sum_{k\ge \ell + 10} 2^{k(\frac12+\gamma_1)}\Big(  \int_{\Sph^2} \int_{[r\simeq 2^k]} \Big \| \one_{[|x|\simeq 2^\ell]} v(x)   L(r-2\omega \cdot x, \omega)  \Big  \|_{ L^2_x}^2  \, dr d\omega   \Big)^{\frac12}   \\
&\quad + \sum_{\ell=0}^\infty   2^{\ell(1+\gamma_1+\gamma_2)}\Big(  \int_{\Sph^2} \int_{[0 <r\les 2^\ell]} \Big \| \one_{[|x|\simeq 2^\ell]} v(x)   L(r-2\omega \cdot x, \omega)  \Big  \|_{ L^2_x}^2  \, dr d\omega   \Big)^{\frac12}    \\
&\les   \sum_{\ell=0}^\infty 2^{\ell(\frac{1}{2}+\gamma_2)} \| \one_{[|x|\simeq 2^\ell]} v(x)\|_2 \sum_{k\ge \ell + 10} 2^{k(\frac12+\gamma_1)}\Big(  \int_{\Sph^2} \int_{[r\simeq 2^k]} \big |    L(r, \omega)  |^2 \, dr d\omega   \Big)^{\frac12}   \\
&\qquad + \sum_{\ell=0}^\infty   2^{\ell(1+\gamma_1+\gamma_2)}  \| \one_{[|x|\simeq 2^\ell]} v(x)   \|_2 \Big(  \int_{\Sph^2} \int_{[|r| \les 2^\ell]}    | L(r, \omega)  |^2 \, dr d\omega   \Big)^{\frac12}  \\
&\les \|v\|_{B^{1+\gamma_1+\gamma_2}} \|V\|_{B^{\frac12+\gamma_1}}
}
as claimed.  
\end{proof} 

The important feature of \eqref{eq:Kest2weighted} is that the weights only accumulate on $v$, but not on $V$, which is the internal function in $K_{1+}$.  This is important since the same $B$-norm  
then appears in the integral on the left-hand side as on the right-hand side of~\eqref{eq:Kest2weighted}, by setting $\gamma_1=\gamma_2$.   Without this feature the algebra formalism developed in the next two
sections would be impossible. We also remark that $\gamma_1=\gamma_2>0$ is needed in the Wiener theorem, to  ensure that conditions \eqref{eq:Sass1}, \eqref{eq:Sass2} hold. 
To be specific, $\gamma_1>0$ gives the decay in $y$ needed to guarantee the asymptotic vanishing in \eqref{eq:Sass2}.

 
\section{Function algebras and $W_{n+}^{\eps}$}
\label{sec:FA}

Next, we establish a framework in which $\oast$ as in Definition~\ref{def:comp} is a bounded operation. This is needed in order to express the 
relation between $T_{1+}^\eps$ and $T_+^\eps$.   The following space is very natural in view of the operators $T_{n+}^\eps$ from the previous sections. 

\begin{definition}
\label{def:Z}
We define the Banach space $Z$  of tempered distributions as 
\EQ{
Z := \{T(x_0, x_1, y)\in \calS'(\R^{9}) \mid \;&   \calF_{y} T(x_{0},x_{1},\eta) \in L^{\infty}_{\eta}L^{\infty}_{x_{1}}L^{1}_{x_{0}} \} 
}
with norm 
\EQ{\lb{eq:Znorm}
\| T\|_Z := \sup_{\eta\in\R^3} \| \mc F_{y}T(x_0, x_1, \eta)\|_{{L}^{\infty}_{x_1}L^{1}_{x_{0}}} 
}
where the $\sup$ is the essential supremum.   We adjoin the identity $I$ to $Z$, which corresponds to the kernel $T=\delta_0(y)\delta_0(x_1-x_0)$.   The operation $\oast$ on $T_{1},T_{2}\in Z$ is defined by 
\EQ{\label{eq:FToast}
(T_{1}\oast T_{2} )(x_0, x_2,y) = \mc F_{\eta}^{-1} \Big[ \int_{\R^{3}}\mc F_{y}T_{1}(x_0, x_1, \eta) \mc F_{y}T_{2}(x_1, x_2, \eta)\, dx_{1} \Big](y)
}
\end{definition}

\begin{lemma}\lb{lm2.4} Let
$Z$ is a Banach algebra under $\oast$ with identity~$I$.
If $V \in L^{3/2, 1}$ then $T_{1+}^\eps$ defined by (\ref{2.16}) is in $Z$ and  
$\mc F_y T_{1+}^\eps$ is given by
\be\lb{2.90}\begin{aligned}
\mc F_y T_{1+}^\eps (x_0, x_1, \eta)  
&=  e^{-ix_1 \eta}\,  R_0(|\eta|^2-i\eps )(x_0, x_1)V(x_0)\,  e^{ix_0 \eta}.
\end{aligned}\ee
Moreover, 
\EQ{\label{eq:T1Zbd}
\sup_{{\eps>0}}\|T_{1+}^\eps\|_{Z}\les \|V\|_{L^{3/2, 1}} \les \|V\|_{\dot B^{\frac12}}
}
If, in addition, $0$ energy is regular for $H = -\Delta+V$ in the sense of Definition~\ref{def:0res}, then $T_+^\eps $ also belongs to $Z$ and
\be\lb{2.37}
(I + T_{1+}^\eps ) \oast (I - T_+^\eps ) = (I - T_+^\eps ) \oast (I + T_{1+}^\eps ) = I.
\ee
\end{lemma}
\begin{proof}
$Z$ is clearly a Banach space. The expressions in brackets in \eqref{eq:FToast} satisfies 
\EQ{
&\sup_{\eta\in\R^{3}}\Big \|  \int_{\R^{3}}\mc F_{y}T_{1}(x_0, x_1, \eta) \mc F_{y}T_{2}(x_1, x_2, \eta)\, dx_{1}  \Big\|_{L^{\infty}_{x_{2}}L^{1}_{x_{0}}} \\
&\le \| \mc F_{y}T_{1}\|_{{L}^{\infty}_{\eta}{L}^{\infty}_{x_1}L^{1}_{x_{0}}} \| \mc F_{y}T_{2}\|_{{L}^{\infty}_{\eta}{L}^{\infty}_{x_2}L^{1}_{x_{1}}} 
= \| T_{1}\|_{Z} \|T_{2}\|_{Z}
}
whence it is a tempered distribution in $\R^{9}$. Therefore, the composition~\eqref{eq:FToast} is well-defined in $Z$ and 
\[
\| T_{1}\oast T_{2}\|_{Z}\le \| T_{1}\|_{Z} \|T_{2}\|_{Z}
\]
so $Z$ is a Banach algebra under $\| \cdot\|_{Z}$.

Formula~\eqref{2.90} is the same as~\eqref{eq:t1'}.  That $T_{1+}^{\eps }\in Z$ and   $T_+^\eps \in Z$ under the $0$ energy condition is a restatement of Lemma~\ref{lem:0regBd}. 
The resolvent identity (\ref{2.29}) implies that 
\EQ{\label{eq:RI2}
 R_0( |\eta|^2-i\eps )V - R_V( |\eta|^2-i\eps )V + R_0( |\eta|^2-i\eps )V R_V( |\eta|^2-i\eps )V  =0
}
whence, with $e^{ix \eta}f(x)=: (M_{\eta}f)(x)$,   
\begin{align}\label{eq:wichtig}
& M_{\eta}^{-1}\, R_0( |\eta|^2-i\eps )(x_0,x_1)V(x_0)\,  M_{\eta} - M_{\eta}^{-1} \,  R_V( |\eta|^2-i\eps )(x_0,x_1)V(x_0)\,  M_{\eta}\\
&+ M_{\eta}^{-1}\,  R_0( |\eta|^2-i\eps )(x_2,x_1)V(x_2)\,  M_{\eta}  \circ M_{\eta}^{-1}\, R_V( |\eta|^2-i\eps )(x_0,x_2)V(x_0)\,  M_{\eta}  =0\nn
\end{align}
where $\circ$ signifies integration.  In view of \eqref{2.90} this means
\EQ{
0= T_{1+}^\eps - T_{+}^\eps + T_{1+}^\eps\oast T_{+}^\eps
}
or $(I + T_{1+}^\eps ) \oast (I - T_+^\eps )=I$. The second identity in \eqref{2.37} holds because the resolvent identity also implies \eqref{eq:RI2} with $R_0$ and $R_V$ reversed:
\EQ{\label{eq:RI3}
 R_0( |\eta|^2-i\eps )V - R_V( |\eta|^2-i\eps )V + R_V( |\eta|^2-i\eps )V R_0( |\eta|^2-i\eps )V  =0
}
and so that same argument as before concludes the proof. 
\end{proof}

The space $Z$ by itself is not sufficient to control the wave operators via Wiener's theorem. This requires other spaces, mainly the algebra~$Y$,  to which we now turn.

\begin{definition}\label{def:VB}
Fix  some number $\frac12\le \sigma<1$, and 
a function $v\in B^{\sigma}$ which does not vanish a.e.\ ($B^{\sigma}$ is the space from Definition~\ref{def:Bspace}). 
We  introduce the following structures depending on~$v$: \begin{itemize}
\item the
 seminormed space $v^{-1}B$ is defined as 
$$
v^{-1}B=\{f\ \text{measurable} \mid v(x) f(x) \in B^{\sigma} \}
$$
with the seminorm $\|f\|_{v^{-1} B}:=\|v f\|_{B^{\sigma} }$.   
\item  Set $X_{x,y}:=L^{1}_{y}v^{-1}B_{x}$. 
Let $Y$ be the space of three-variable kernels
\be\lb{2.31}\begin{aligned}
Y &:= \Big\{T(x_0, x_1, y) \in Z \mid \; \forall f \in  L^\infty\\
& (fT)(x_{1}, y):= \int_{\R^3} f(x_0) T(x_0, x_1, y) \dd x_0 \in X_{x_{1},y} \Big\},
\end{aligned}\ee
with norm
\EQ{\lb{eq:Ynorm}
\|T\|_Y &:= \|T\|_Z  + \|T\|_{\B(v^{-1}B_{x_0}, X_{x_{1},y} )} 
}
We   adjoin  an identity element to $Y$, in the form of
\be\lb{ident}
I(x_0, x_1, y) = \delta_{x_0}(x_1) \delta_0(y) = \delta_{x_1}(x_0) \delta_0(y).
\ee
\end{itemize}
\end{definition}

While we keep this definition more general with regard to the function $v$, in our applications below we will set $v=V$, the potential in $H=-\Delta+V$.
Since $v\in B^{\sigma}$, we have 
\EQ{\label{eq:XX0}
L^{\infty}\subset v^{-1}B,\qquad \|f\|_{v^{-1}B}\le \|v\|_{B^{\sigma}}\|f\|_{\infty}.
}  
Moreover, $L^{\infty}$ is dense in $v^{-1}B$. 
The spaces $v^{-1}B$ and $Y$ depend on $\sigma$, but so as not to overload the notation we suppress this dependence. 
Note that the $x_0$-integral in \eqref{2.31} is well-defined for any $f\in L^\infty$ due to   $T\in Z$, and that this integration produces a tempered distribution in the variables $(x_{1},y)$. 
The condition is then that the Schwartz kernel of this distribution satisfies a bound of the form, for all  $f\in L^{\infty}$,  
\EQ{
\| fT \|_{ X_{x_{1},y} } = \int_{\R^3} \|v(x_{1}) (fT)(x_{1}, y)\|_{B^{\sigma}_{x_{1}} } \dd y  \le A \|vf\|_{B^{\sigma}} 
}
for some finite constant $A$. 

We first record some formal properties of these spaces. 
In what follows, the parameter $\frac12\le \sigma<1$ will be kept fixed. In principle we would like to set $\sigma=\frac12$ which is optimal. But in order to obtain decay in $y$ for the $Y$ algebra,
which is needed in Wiener's theorem in the next section, we require $\sigma>\frac12$, cf.~Lemma~\ref{lem:K1weighted}.

\begin{lemma}\label{lem:Ynorms}
For any $T\in Y$
\EQ{\label{eq:Tnorm+}
 \|T\|_{Y} \les  \|T\|_{L^{1}_{y}\,\B(v^{-1}B_{x_{0}}, L^{\infty}_{x_{1}})}
}
provided the right-hand side is finite. 
\end{lemma}
\begin{proof}
Apply the embedding $L^{\infty}\embed v^{-1}B$ and Minkowski's inequality (to pull out the $L^{1}_{y}$ norm). 
\end{proof}

The algebras used in \cite{bec,becgol}
 have the structure of $L^{1}_{y}$ convolution algebras, taking values in the bounded operators on some Banach space~$X$, cf.~the right-hand side of~\eqref{eq:Tnorm+}.   In~\cite{bec,becgol} it suffices to consider the {\em one-dimensional} Fourier transform of $R_{0}(\lambda^{2}+i0)(x_{0},x_{1})$ relative to~$\lambda$, which is a measure supported on the sphere of radius $|x_{0}-x_{1}|$ in $\R^{3}$.  However, because of the phases $e^{\pm i x\cdot\eta}$ in~\eqref{eq:wichtig}, the dependence on~$\eta$ is truly three-dimensional  and the Fourier transform of~\eqref{2.90} relative to~$\eta$ is not a measure.

\begin{lemma} \label{lem:6prop}
The spaces in the previous definition possess the following properties:
\begin{enumerate}[i)]
\item Let $f\in v^{-1}B$. Then $\|f\|_{v^{-1}B}=0$, if and only if $f=0$ a.e.~on the set $\{v\ne0\}$. Restricting all functions in $v^{-1}B$ to the set $\{v\ne0\}$ turns $v^{-1}B$ into a Banach space. 
$L^1\cap L^\infty$ is dense in $v^{-1}B$, and so bounded compactly supported functions are dense in $v^{-1}B$.   $X^{0}_{x,y}:= L^{1}_{y}L^{\infty}_{x}$ is dense in~$X_{x,y}$. 
\item The space  $Y$ is a Banach space.  
 It is invariant under multiplication by bounded functions of $y$ and under changes of variable in $y$ that preserve the $L^1_y$ norm. The   $Y$ norm is invariant under translation in $y$. If $\chi$ is a Schwartz function in $\R^{3}$, then 
\[
\| T\ast \chi\|_{Y}\le \|T\|_{Y}\|\chi\|_{1}
\]
where $\ast$ denotes convolution relative to $y$ in the distributional sense. 
\item For $\frak X \in X^{0}$, define the \emph{contraction} of $T\in Y$ by $\frak X$ via 
\be\lb{contractie}
(\frak X T)(x, y) := \int_{\R^6} \frak X(x_0, y_0) T(x_0, x, y-y_0) \dd x_0 \dd y_0.   
\ee
Then $\frak X T\in X_{x,y}$ and $\|\frak X T\|_{X}\le \|T\|_{Y}\|\frak X\|_{X}$.  The right-hand side of \eqref{contractie} is to be interpreted on the Fourier side as
\EQ{\lb{eq:mcFXT}
\mc F^{-1}_{\eta}\Big [ \int_{\R^3} \mc F_{y_{0}} \frak X(x_0, \eta) \mc F_{y_{0}}T(x_0, x,\eta) \dd x_0 \Big](y)
}
The integral is absolutely convergent, and the inverse Fourier transform relative to $\eta$ is a tempered distribution.  By density of $X^{0}$ in $X$, 
the contraction $\frak X T$ is well-defined for any $\frak X\in X$. 
 \item 
Any  $T\in Y$ possesses  a distributional Fourier transform in the $y$ variable $\big(\mc F_y T(\eta)\big)(x_0, x_1)$, and 
\EQ{\label{eq:mcFY}
\mc F_y T(\eta) \in \cB(L^{\infty})\cap \B(v^{-1}B) =:\mc FY
}
where $\B(E)$ denotes the bounded linear operators on the (semi)normed space $E$. 
Moreover, 
\EQ{\lb{eq:FyTbd}
\sup_{\eta}\|\mc F_{y} T(\eta)\|_{\mc FY}=\| \mc F_y T\|_{L^{\infty}_{\eta} \B(v^{-1}B)}+\| \mc F_y T\|_{L^{\infty}_{\eta} \B(L^{\infty})}\les \|T\|_{Y}
}
\item 
A kernel of the form $S(x_0, x_1) \chi(y)$ where $\chi \in L^1$ and $S \in \mc FY$ belongs to  $Y$ and 
$$
\|S(x_0, x_1) \chi(y)\|_Y \les \|S\|_{\mc FY} \|\chi\|_{L^1}.
$$
More generally, $L^{1}_y \mc FY \subset Y$.
\item Let  $U\in \mc FY$.
Then for any $T=T(x_0,x_1,y)\in Y$ one has 
\[
(U\circ T)(x_0,x_2,y) := \int_{\R^3} U(x_1,x_2) T(x_0,x_1,y)\, dx_1 \in Y
\]
and
\[
\| U\circ T\|_Y \les \|U\|_{\mc FY} \|T\|_Y
\]
Analogous results hold for $T\circ U$. 
\end{enumerate}
\end{lemma}
\begin{proof}
The properties of $v^{-1}B$ follow from simple measure theory, and we skip the details.  By \eqref{eq:XX0},
\[
\| \frak X\|_{L^{1}_{y}v^{-1}B_{x}} \le \|v\|_{B^{\sigma}}\| \frak X\|_{L^{1}_{y}L^{\infty}_{x}}
\]
and so $X^{0}$ embeds continuously into~$X$, and is dense in~$X$. 

The $Z$ component in~\eqref{eq:Ynorm} guarantees that $\|\cdot\|_{Y}$ is truly a norm. We further note that $\B(v^{-1}B_{x_0}, L^1_y v^{-1} B_{x_{1}} )$ is a Banach space provided we restrict both $x_{0}$ and $x_{1}$ to $\{v\ne0\}$. Thus, $Y$ is complete relative to both components of the $Y$-norm.  

For iii), one has 
\EQ{
\| (\frak X T)(x, y)\|_{X_{x,y}} &= \Big \| \int_{\R^3} \big[ \frak X(\cdot, y_0)T(\cdot)\big] (x, y-y_0) \dd y_0 \Big\|_{X_{x,y}} \\
&\le \|T\|_{Y} \int_{\R^{3}} \|\frak X(\cdot,y_{0})\|_{v^{-1}B}\, dy_{0}
}
as claimed. For iv), first note that   $T\in Y\subset Z$ has the property that 
$$\mc F_{y}T(x_{0},x_{1},\eta) \in  L^{\infty}_{\eta}L^{\infty}_{x_{1}}L^{1}_{x_{0}}=   L^{\infty}_{\eta} \B(L^{\infty}_{x_{0}},L^{\infty}_{x_{1}})$$
Second, for any Schwartz function $f$,
\EQ{\lb{eq:FyT}
&\sup_{\eta\in\R^{3}} \Big\|\int_{\R^{3}} \mc F_{y}T(x_{0},x_{1},\eta)f(x_{0})\, dx_{0}\Big\|_{v^{-1}B_{x_{1}}} \\
&\le \Big \| \int_{\R^{3}} T(x_{0},x_{1},y) f(x_{0})\, dx_{0} \Big\|_{v^{-1}B_{x_{1}} L^{1}_{y}} \\
&\le \|T\|_{\B(v^{-1} B_{x_0},  L^{1}_{y}v^{-1}B_{x_{1}})} \|f\|_{v^{-1}B} \\
&\le  \|T\|_{\B(v^{-1} B_{x_0}, X_{x_{1},y})} \|f\|_{v^{-1}B}
} 
Properties v) and vi) are evident from the definitions. 
\end{proof}
 
The space $Y$ is by definition the space of kernels $T\in Z$ so that the contraction $fT$ for any $f\in v^{-1}B$ lies in $X$, where we define the contraction by~\eqref{2.31}. 
This is  what Lemma~\ref{lem:K1} expresses for $T_{1+}^{\eps}$ with $v=V$. 

\begin{corollary}
\label{cor:T1Y}
Let $V$ be Schwartz and apply Definition~\ref{def:VB} with $v=V$, the potential. Then for every $\eps>0$ we have $T_{1+}^{\eps}\in Y$ (where $\sigma\in [\frac12,1)$ is arbitrary but fixed) and
\EQ{
\sup_{\eps>0} \| T_{1+}^{\eps} \|_{Y} \les \| V\|_{B^{\frac12+\sigma}}
}
\end{corollary}
\begin{proof}
By \eqref{eq:T1Zbd} we have 
\[
\sup_{\eps>0} \| T_{1+}^{\eps}\|_{Z}\les \|V\|_{B^{\frac12}}
\]
Lemma~\ref{lem:K1} implies that 
\EQ{\label{eq:fTBB}
\sup_{\eps>0}\Big\| \int_{\R^{3}} f(x_{0}) T_{1+}^{\eps}(x_{0},x,y)\, dx_{0}\Big\|_{X_{x,y}} &\les   \|fV\|_{B^{\sigma}} \|V\|_{B^{\frac12+\sigma}}
}
which concludes the proof that $T_{1+}^{\eps}\in Y$ with given $\sigma\in [\frac12, 1)$. 
\end{proof}

By the same proof, Lemma~\ref{lem:K1} evidently allows allows for a stronger conclusion, namely 
\EQ{\label{eq:T1 ext}
\sup_{\eps>0} \| T_{1+}^{\eps} \|_{Y^{e}} \les \| V\|_{B^{\frac12+\sigma}}
}
where $Y^{e}$ is the {\em extended} (with respect to the {\em norm}) space $Y$. It is defined as above, but using $X^{e}:=X\cap L^{\infty}_{x}L^{1}_{y}$ instead of~$X$.
 We do not include the $L^{\infty}_{x}L^{1}_{y}$-norm in our construction of $X$ and $Y$ above as this would invalidate the condition~\eqref{eq:Sass2T}, which is crucial for the Wiener theorem.  However, once we have applied the Wiener theorem in the larger algebra~$Y$, we are then able to control the $L^{\infty}_{x}L^{1}_{y}$-norm.  

We now state the algebra property of~$Y$ (which also holds for $Y^{e}$). 

\begin{lemma}\lb{yalg}
$Y$ defined by (\ref{2.31}) is a  Banach algebra with the operation $\oast$ defined in the ambient algebra~$Z$.  
\end{lemma}
\begin{proof}
The fact that $\oast$ is associative (and non-commutative) is clear in $Z$, and the unit element is given by (\ref{ident}). Since $Y \subset Z$, the same is true in $Y$.

 The definitions of $X$ and $Y$ imply that each contraction $\frak X T$ (see (\ref{contractie})) is in $X$ and $\|\frak X T\|_X \les \|\frak X\|_X \|T\|_Y$.
 We have 
\EQ{\label{eq:T3oast}
&\int_{\R^3} f(x_0) T_3(x_0, x_2, y) \dd x_0 = \\
&= \int_{\set R^9} f(x_0) T_1(x_0, x_1, y_1) T_2(x_1, x_2, y-y_1) \dd x_1 \dd y_1 \dd x_0.
}
As in the case of \eqref{contractie}, the $y$-integral is to be understood in the distributional Fourier sense. 
Integrating in $x_0$, we obtain an expression of the form $\frak X T_2$ for $\frak X \in X$ with $\|\frak X\|_X \les \|f\|_{V^{-1} B} \|T_1\|_Y$. Then $\frak X T_2$ belongs to $X$ as stated above and has a norm  at most $\les \|f\|_{V^{-1} B} \|T_1\|_Y \|T_2\|_Y$. Thus, $T_3 = T_1 \oast T_2 \in Y$ and
$$
\|T_1 \oast T_2\|_Y \le C \|T_1\|_Y \|T_2\|_Y
$$
with some absolute constant $C$. Multiplying the norm by $C$ removes this constant from the previous inequality, and so $Y$ is an algebra under this new norm. 
\end{proof}

Thus, provided $I + T_{1+}^\eps $ is invertible in $Y$, hence in $Z$, its inverse will be $I - T_+^\eps $ both in $Z$ and in $Y$, hence we obtain that $T_+^\eps  \in Y$.

\begin{prop}\lb{prop:Tngn}
Let $V$ be a Schwartz potential. Then $T_{n+}^{\eps}\in Y$ (where $\sigma\in [\frac12,1)$ is arbitrary but fixed) for any $n\ge1$ and $\eps>0$ and 
\EQ{\label{eq:WnY}
\sup_{\eps>0} \| T_{n+}^{\eps}\|_{Y}\le C^{n}\|V\|_{B^{\frac12+\sigma}}^{n}
}
with some absolute constant $C$.  Moreover, for all Schwartz functions $f$ one has 
\EQ{\label{eq:gn}
(W_{n+}^{\eps}f)(x) = \int_{\Sph^2 }\int_{\R^3} g_n^{\eps}(x,dy,\omega) f(S_\omega x-y)\,  \sigma(d\omega)
}
where  for fixed $x\in\R^3$, $\omega\in \Sph^2 $ the expression $g_n^{\eps}(x,\cdot,\omega)$ is a measure satisfying 
\EQ{\lb{eq:gnmeas}
\sup_{\eps>0} \int_{\Sph^2 } \| g_n^{\eps}(x,dy,\omega)\|_{\mes_{y} L^\infty_x}  \, d\omega \le  C^{n}\|V\|_{B^{\frac12+\sigma}}^{n}
}
where $\|\cdot \|_{\mes}$ refers to the total variation norm of Borel measures.  Relations \eqref{eq:gn} and \eqref{eq:gnmeas} remain valid if $V\in B^{\frac12+\sigma}$. 
\end{prop}
\begin{proof}
By Lemma~\ref{lem:Tast}, $T_{n+}^{\eps} = T_{1+}^{\eps}\oast T_{(n-1)+}^{\eps}$. Corollary~\ref{cor:T1Y} and Lemma~\ref{yalg} imply~\eqref{eq:WnY} by induction. 

From Lemma~\ref{lemma2.1}, using the contraction formalism from above,  and identifying the operator $W_{n+}^{\eps}$ with its kernel one has 
\EQ{\lb{eq:1Tn}
W_{n+}^{\eps} &= (-1)^{n} \one_{\R^{3}}T_{n+}^{\eps} = (-1)^{n} \one_{\R^{3}}(T_{(n-1)+}^{\eps}\oast T_{1+}^{\eps} )\\
&= -((-1)^{n-1} \one_{\R^{3}} T_{(n-1)+}^{\eps}) T_{1+}^{\eps} = - W_{(n-1)+}^{\eps} T_{1+}^{\eps}
}
The notation in the second line contraction of a kernel in $Y$ by an element of $X$; this follows again by induction starting from  $W_{0+}^{\eps}=\one_{\R^{3}}$ via~\eqref{contractie}. 
Strictly speaking, we have so far considered contractions only against Schwartz functions. But  $\one_{\R^{3}}$  is the limit in the space $V^{-1}B$ of smooth bump functions $\chi(\cdot/R)$
as $R\to\infty$ (where $\chi$ is smooth compactly supported and $\chi=1$ on the unit ball). By the boundedness of $T_{n+}^{\eps}$ in $Y$ it follows that the right-hand side of \eqref{eq:1Tn} is well-defined in $Y$.  Thus,  by the first equality sign in~\eqref{eq:1Tn}, 
\EQ{\label{eq:WnX}
\sup_{\eps>0} \| W_{n+}^{\eps}\|_{X}\le \|\one_{\R^{3}}\|_{V^{-1}B} \sup_{\eps>0} \|T_{n+}^{\eps}\|_{Y} \le C^{n}\|V\|_{B^{\frac12+\sigma}}^{n+1}
}
We denote the kernel of $W_{1+}^{\eps}$ by $\frak X_{V}^{\eps}$, where $V$ is the potential.  Thus,
\EQ{\label{eq:frakXV}
\frak X_{V}^{\eps} (x,y) = - \int_{\R^{3}}  T_{1+}^{\eps} (x_{0},x,y)\, dx_{0} =  -(\one_{\R^{3}} T_{1+}^{\eps})(x,y) \in X
}
By the final equality sign in \eqref{eq:1Tn}, 
\EQ{\label{eq:Wneps}
W_{n+}^{\eps}(x,y) &= -\int_{\R^{6}} W^{\eps}_{(n-1)+}(x',y')T_{1+}^{\eps}(x',x,y-y') \, dx'dy' \\
&= \int_{\R^{3}} \frak X^{\eps}_{f^{\eps}_{y'}V}(x,y-y')\, dy'
}
Here we wrote $f^{\eps}_{y'}(x')=W^{\eps}_{(n-1)+}(x',y')$ and we used \eqref{eq:fV1}. We also assumed that $ f^{\eps}_{y'}(x')$ is Schwartz in $x'$ to make this calculation rigorous. Later we shall remove this assumption by approximation. 

We now invoke the representation from Corollary~\ref{cor:K1+}. Specifically, by \eqref{eq:g1} there exists $g_{1,f^{\eps}_{y'}}^{\eps}(x,dy,\omega)$  so that for every 
$\phi\in \mc S$ one has
\EQ{\label{eq:g1*}
(\frak X^{\eps}_{f^{\eps}_{y'}V}\; \phi)(x) = \int_{\Sph^2 }\int_{\R^3} g_{1,f^{\eps}_{y'}}^{\eps} (x,dy,\omega) \phi(S_\omega x-y)\,  \sigma(d\omega)
}
where  for fixed $x\in\R^3$, $\omega\in \Sph^2 $ the expression $g_{1,f^{\eps}_{y'}}^{\eps}(x,\cdot,\omega)$ is a measure satisfying 
\EQ{\lb{eq:gnmeas'}
\sup_{\eps>0} \int_{\Sph^2 }  \| g_{1,f^{\eps}_{y'}}^{\eps}(x,dy,\omega)\|_{\mes_{y} L^\infty_x}  \, d\omega &\le C\|f^{\eps}_{y'}V\|_{B^{\frac12}} \le C\|f^{\eps}_{y'}V\|_{B^{\sigma}}\\
&= C  \| W^{\eps}_{(n-1)+}(x',y') \|_{V^{-1}B_{x'}}    
}
Thus
\begin{align} 
(W^{\eps}_{n+}\phi)(x) &= \int_{\R^{3}} W_{n+}^{\eps}(x,y) \phi(x-y)\, dy \nn  \\
&= \int_{\R^{6}} \frak X^{\eps}_{f^{\eps}_{y'}V}(x,y-y') \phi(x-y)\, dy dy'  \nn  \\
&=   \int_{\R^{6}} \frak X^{\eps}_{f^{\eps}_{y'}V}(x,y) \phi(x-y-y')\, dy dy'  \nn  \\
& = \int_{\R^{3}} (\frak X^{\eps}_{f^{\eps}_{y'}V} \phi)(x-y')\,dy' \lb{eq:I} \\
&= \int_{\R^{3}} \int_{\Sph^2 }\int_{\R^3} g_{1,f^{\eps}_{y'}}^{\eps} (x-y',dy,\omega) \phi(S_\omega (x-y')-y)\,  \sigma(d\omega) \,dy' \nn \\
&=  \int_{\Sph^2 }\int_{\R^3} \Big[ \int_{\R^{3}} g_{1,f^{\eps}_{y'}}^{\eps} (x-y',d(y-S_{\omega}y'),\omega) \, dy'\Big] \phi(S_\omega x-y)\,  \sigma(d\omega)   \nn 
\end{align}
The expressions in brackets is the kernel we seek, i.e., 
\EQ{\lb{eq:II}
g_{n}(x,dy, \omega) := \int_{\R^{3}} g_{1,f^{\eps}_{y'}}^{\eps} (x-y',d(y-S_{\omega}y'),\omega) \, dy'
}
This object is a measure in the $y$-coordinate and we have the representation
\EQ{\lb{eq:III}
(W^{\eps}_{n+}\phi)(x) &= \int_{\Sph^{2}} \int_{\R^{3}}g_{n}(x,dy, \omega) \phi(S_\omega x-y)\,  \sigma(d\omega) 
}
as well as the size control uniformly in $\eps>0$ 
\EQ{\lb{eq:IV}
 & \int_{\Sph^2 }  \| g_n^{\eps}(x,dy,\omega)\|_{\mes_{y} L^\infty_x}  \, d\omega  \\
& = \int_{\Sph^2 }\int_{\R^3}    \big\| g_{1,f^{\eps}_{y'}}^{\eps} (x-y',d(y-S_{\omega}y'),\omega) \big \|_{\mes_{y} L^\infty_x}   \,  dy'  d\omega\\
&= \int_{\Sph^2 }\int_{\R^3}   \big\| g_{1,f^{\eps}_{y'}}^{\eps} (x,dy,\omega) \big \|_{\mes_{y} L^\infty_x}   \,  dy'  d\omega   \\
&\le    C  \int_{\R^3}  \| W^{\eps}_{(n-1)+}(x',y') \|_{V^{-1}B_{x'}} \,  dy' \\
&\le  C   \| W^{\eps}_{(n-1)+} \|_{X} \le C^{n}\|V\|_{B^{\frac12+\sigma}}^{n}
}
by \eqref{eq:WnX}
as desired. Recall that we assumed that $f^{\eps}_{y'}(x')$ is a Schwartz function. To remove this assumption,  we can make $\| W^{\eps}_{(n-1)+}(x',y')-\tilde f^{\eps}_{y'}(x') \|_{X}$ arbitrarily small with a Schwartz function $\tilde f^{\eps}_{y'}(x')$ in $\R^{6}$. Then the previous calculation  shows that 
\[
\int_{\Sph^2 }  \| g_n^{\eps}(x,dy,\omega)- \tilde g_n^{\eps}(x,dy,\omega)\|_{\mes_{y} L^\infty_x}  \, d\omega 
\]
can be made as small as we wish where $\tilde g_n^{\eps}(x,dy,\omega)$ is the function generated by $\tilde f^{\eps}_{y'}(x')$. Passing to the limit concludes the proof. 

To remove the assumption that $V$ be a Schwartz function, we approximate $V\in B^{\frac12+\sigma}$ by Schwartz functions in the norm $\|\cdot\|_{B^{\frac12+\sigma}}$. 
We achieve convergence of of the functions $g_{n}$ by means of \eqref{eq:gnmeas} and of the kernels $W_{n+}^{\eps}$ themselves by means of~\eqref{eq:WnX}. To be specific, denoting by $\widetilde W^{\eps}_{n+}$ and  $\tilde g_{n}$ the quantities corresponding to the potential $\tilde V$,
taking differences yields 
\EQ{\nn 
& \| \widetilde  W^{\eps}_{n+} -  W^{\eps}_{n+} \|_{X} + \int_{\Sph^2 } \| g_n^{\eps}(x,dy,\omega)-\tilde g_n^{\eps}(x,dy,\omega)\|_{\mes_{y} L^\infty_x}  \, d\omega \\
& \le C^{n} \| V-\tilde V\|_{B^{\frac12+\sigma}}(\|V\|^{n-1}_{B^{\frac12+\sigma}}+ \|\tilde V\|_{B^{\frac12+\sigma}}^{n-1})
}
uniformly in $\eps>0$.  
\end{proof}

If the potential is small, then we can sum the geometric series which arises in the previous proposition and therefore obtain the structure theorem with explicit bounds in that case. For large potentials we now introduce the Wiener formalism.


\section{Wiener's theorem and the proof of the structure formula}
\label{sec:Wiener}

To set the stage for the  technique of this section, we first recall the following classical result by Norbert Wiener.  It concerns the invertibility problem of $\delta_0+f$ in
the algebra $L^1(\R^{d})$ with unit (we formally adjoin $\delta_0$ to $L^1(\R)$). Here the dimension $d\ge1$ is arbitrary.  
Throughout this section, we let $\chi$ be a Schwartz function with $\hat{\chi}(\xi)=1$ on $|\xi|\le 1$
and $\hat{\chi}(\xi)=0$ on $|\xi|\ge 2$. Then $\int\chi=\hat{\chi}(0)=1$.  Further,  $\chi_{R}(x)=R^{d}\chi(Rx)$, so that $\what{\chi_{R}}=\what{\chi}(R^{-1}\xi)$. We can further assume that $\chi$ is radial. 

\begin{prop}\label{prop:Wiener}
Let $f\in L^1(\R^{d})$. Then there exists $g\in L^1(\R^{d})$ so that 
\EQ{
\label{eq:fghat}
(1+\hat{f}\,)(1+\hat{g})=1 \text{\ \ on\ \ }\R^{d}
}
  if and only of  $1+\hat{f}\ne0$ everywhere. 
Equivalently, there exists $g\in L^1(\R^{d})$ so that 
\EQ{\lb{eq:FP}
(\delta_0+f)\ast (\delta_0+g)=\delta_0
}
 if and only of  $1+\hat{f}\ne0$ everywhere on $\R^{d}$.  The function $g$ is  unique. 
\end{prop}
\begin{proof}
The idea is to find local solutions of \eqref{eq:fghat} and then patch them together using a partition of unity to obtain a single function $g\in L^{1}$. 
First, we will find $g_{0}\in L^{1}$ so that \eqref{eq:fghat} holds for all $|\xi|\ge R$, $R$ large.  We select $R\ge 1$ so large that 
\EQ{\label{eq:L1klein}
\| f - \chi_{R}\ast f\|_{1}= \| (\delta_{0}-\chi_{R})\ast f\|_{1}<\frac12
}
In particular,  $\|(1-\what{\chi_{R}})\hat{f}\|_{\infty}<\frac12$. Set $f_{0}:=(\delta_{0}-\chi_{R})\ast f$ and note that 
\[
(\delta_{0}+f_{0})^{-1} = \delta_{0}- f_{0}+f_{0}\ast f_{0} - f_{0}\ast f_{0}\ast f_{0} + \ldots
\]
as a norm convergent series in $L^{1}$, by \eqref{eq:L1klein}. This means that $(\delta_{0}+f_{0})^{-1} =\delta_{0}+ F_{0}$, $F_{0}\in L^{1}$. 
Define 
\[
g_{0}:= - (\delta_{0}-\chi_{R})\ast f \ast (\delta_{0}+ F_{0}) \in L^{1}
\]
which implies that 
\EQ{\nn
\what{g_{0}}:=  -\frac{(1-\what{\chi_{R}})\hat{f}}{1+ (1-\what{\chi_{R}})\hat{f}}
}
or equivalently
\EQ{\nn
(1+(1-\what{\chi_{R}})\hat{f}) (1+\what{g_{0}}) =1
}
which means that 
\EQ{\label{eq:farsol}
(1+\hat{f}(\xi)) (1+\what{g_{0}}(\xi))=1 \qquad \forall\; |\xi|\ge 2R
}
By construction, $g_{0}\in L^{1}$ is therefore a solution of \eqref{eq:fghat} on $|\xi|\ge 2R$. 

As a second step, we need to find $g_{1}\in L^{1}$ so that 
\EQ{\label{eq:nearsol}
(1+\hat{f}(\xi)) (1+\what{g_{1}}(\xi))=1 \qquad \forall\; |\xi|\le 3R
}
This will then easily finish the proof. Indeed, 
let $\psi_{0},\psi_{1}$ be Schwartz functions with the property that $\psi_{1}(\xi)=1$ if $|\xi|\le 2R$ and $\psi_{1}(\xi)=0$ if $|\xi|>3R$. 
Then set $\psi_{0}=1-\psi_{1}$, and let $\what{\phi_{0}}=\psi_{0}$, $\what{\phi_{1}}=\psi_{1}$.  Then 
\[
g:= \phi_{0}\ast g_{0}+\phi_{1}\ast g_{1} \in L^{1}
\]
solves the full equation \eqref{eq:fghat}. Indeed, \eqref{eq:farsol} and \eqref{eq:nearsol} imply that 
\EQ{\lb{eq:pou}
& 1 = \psi_{0}(\xi) + \psi_{1}(\xi) \\
& = (1+\hat{f}(\xi)) (1+\what{g_{0}}(\xi)) \psi_{0}(\xi) + (1+\hat{f}(\xi)) (1+\what{g_{1}}(\xi)) \psi_{1}(\xi) \\
& =  (1+\hat{f}(\xi)) ( 1 + \psi_{0}(\xi) \what{g_{0}}(\xi) + \psi_{1}(\xi) \what{g_{1}}(\xi) )  = (1+\hat{f}(\xi)) ( 1+ \hat{g}(\xi))
}
for all $\xi\in \R^{d}$. 

To find $g_{1}$ which satisfies \eqref{eq:nearsol}, we solve \eqref{eq:fghat} near any $\xi_{0}\in \R^{d}$ with $|\xi_{0}|\le 3R$. 
As before, we then patch up these local solutions by means of a partition of unity. 
Define, for any $1>\eps>0$,  
\[
\omega_{\eps,\xi_{0}}(x) = e^{ix\cdot\xi_{0}} \eps^{d} \chi(\eps x) 
\] 
or 
\[
\what{\omega_{\eps,\xi_{0}}}(\xi) = \what{\chi}(\eps^{-1}(\xi-\xi_{0}))
\]
We first claim that
\EQ{
\label{eq:claim} 
\sup_{\xi_{0}\in\R^{d}} \big\| f \ast \omega_{\eps,\xi_{0}} - \hat{f}(\xi_{0}) \omega_{\eps,\xi_{0}}\big \|_{1} \to 0\text{\ \ as\ \ }\eps\to0
}
In fact, one has 
\EQ{
& f \ast \omega_{\eps,\xi_{0}} (x) - \hat{f}(\xi_{0}) \omega_{\eps,\xi_{0}} (x) \\
& = \int_{\R^{d}} f(y) e^{i(x-y)\cdot\xi_{0}} \eps^{d} [\chi(\eps(x-y)) - \chi(\eps x) ]\, dy
}
whence
\EQ{
& \big\| f \ast \omega_{\eps,\xi_{0}} (x) - \hat{f}(\xi_{0}) \omega_{\eps,\xi_{0}} (x) \big\|_{L^{1}_{x}} \\
& = \int_{\R^{d}} |f(y)|  \|  \eps^{d} [\chi(\eps(x-y)) - \chi(\eps x) ] \|_{L^{1}_{x}}\, dy \\
&= \int_{\R^{d}} |f(y)|  \|   \chi(\cdot-\eps y) - \chi(\cdot)  \|_{L^{1}_{x}}\, dy 
}
The right-hand side here tends to $0$ as $\eps\to0$ by the Lebesgue dominated convergence theorem, and so \eqref{eq:claim} holds. 
Therefore, we may take $\eps$ small enough such that 
\EQ{\lb{eq:Wiener}
 & (1+\hat{f}(\xi_{0}))\delta_{0} +  f \ast \omega_{\eps,\xi_{0}} - \hat{f}(\xi_{0}) \omega_{\eps,\xi_{0}}   \\
 & = (1+\hat{f}(\xi_{0})) \big[ \delta_{0} +  (1+\hat{f}(\xi_{0}))^{-1} \big( f \ast \omega_{\eps,\xi_{0}} - \hat{f}(\xi_{0}) \omega_{\eps,\xi_{0}}  \big)\big] 
}
is invertible for all $\xi_{0}\in \R^{d}$ (in fact, we only need $|\xi_{0}|\le 3R$). This follows from 
\[
m:=\inf_{\xi_{0}\in \R^{d}} |1+\hat{f}(\xi_{0}) | >0
\]
and so the second term in the bracket of \eqref{eq:Wiener} satisfies 
\EQ{
\sup_{\xi_{0}\in \R^{d}}\big\|  (1+\hat{f}(\xi_{0}))^{-1} \big( f \ast \omega_{\eps,\xi_{0}} - \hat{f}(\xi_{0}) \omega_{\eps,\xi_{0}}  \big) \big\|_{1} \le \frac12
}
for $\eps>0$ small enough.   Fix such an $\eps>0$. Then  for all $\xi_{0}\in \R^{d}$, 
\EQ{\lb{eq:Wiener2}
 & \big[(1+\hat{f}(\xi_{0}))\delta_{0} +  f \ast \omega_{\eps,\xi_{0}} - \hat{f}(\xi_{0}) \omega_{\eps,\xi_{0}}\big]^{-1}   \\
 & = (1+\hat{f}(\xi_{0}))^{-1} \big( \delta_{0} +  H_{\xi_{0}}\big)
}
where $H_{\xi_{0}}\in L^{1}$, $\|H_{\xi_{0}}\|_{1}\le 1$.  Let $\Omega_{\eps,\xi_{0}}$ be defined as 
\[
\what{\Omega_{\eps,\xi_{0}}}(\xi) = \what{\chi}(2\eps^{-1}(\xi-\xi_{0}))
\]
By construction, $\Omega_{\eps,\xi_{0}} \ast   \omega_{\eps,\xi_{0}} = \Omega_{\eps,\xi_{0}}$.  Define
\EQ{\label{eq:schluessel}
g_{\xi_{0}} := -   (1+\hat{f}(\xi_{0}))^{-1} f\ast \Omega_{\eps,\xi_{0}} \ast \big( \delta_{0} +  H_{\xi_{0}}\big) \in L^{1}
}
Then 
\EQ{
\what{g_{\xi_{0}}} &= -\frac{  \hat{f} \,  \what{ \Omega_{\eps,\xi_{0}}} }{  1+\hat{f}(\xi_{0})  + (\hat{f} - \hat{f}(\xi_{0}))\what{\omega_{\eps,\xi_{0}}}   } \\
&= -\frac{  \hat{f} \, \what{ \Omega_{\eps,\xi_{0}}} }{  1+\hat{f}(\xi_{0}) \what{\omega_{\eps,\xi_{0}}} + (\hat{f} - \hat{f}(\xi_{0}))\what{\omega_{\eps,\xi_{0}}}   }\\
& =  -\frac{  \hat{f}   \, \what{ \Omega_{\eps,\xi_{0}}} }{ 1 + \hat{f}\; \what{\omega_{\eps,\xi_{0}}}   }
}
The  fraction in the last line is well-defined since on the support of the numerator the cut-off function in the denominator satisfies $\what{\omega_{\eps,\xi_{0}}} =1$. 
In particular, if $\what{ \Omega_{\eps,\xi_{0}}} (\xi)=1$, then 
\EQ{\label{eq:gxi0}
(1+\hat{f}(\xi)) (1+\what{g_{\xi_{0}}}(\xi))=1
}
In other words, we have solved \eqref{eq:fghat} locally near $\xi_{0}$. 
Covering the ball $|\xi|\le 3R$ by finitely many balls of radius $\eps/2$ and summing up these local solutions by means of a subordinate partition of unity as in \eqref{eq:pou} concludes the proof. To be specific, let $\{\phi_{j}\}_{j=1}^{N}$ be Schwartz functions so that $\sum_{j=1}^{N}\what{\phi_{j}}(\xi)=1$ for all $|\xi|\le 3R$.  Moreover, if $\what{\phi_{j}}(\xi)\ne0$,  then $\what{ \Omega_{\eps,\xi_{j}}} (\xi)=1$ for some $\xi_{j}$ with $|\xi_{j}|\le 3R$.  Now set
\EQ{\label{eq:sumg1}
g_{1}:=   \sum_{j=1}^{N}  \phi_{j} \ast g_{\xi_{j}}
}
By construction, 
\EQ{
1&= \sum_{j=1}^{N}\what{\phi_{j}}(\xi) = \sum_{j=1}^{N}\what{\phi_{j}}(\xi) (1+\hat{f}(\xi)) (1+\what{g_{\xi_{0}}}(\xi)) \\
&= (1+\hat{f}(\xi)) (1+ \what{g_{1}}(\xi))
}wq
if $|\xi|\le 3R$, and so $g_{1}\in L^{1}$ is a solution of  \eqref{eq:nearsol}.
\end{proof}

The main goal in this section is to formulate and apply a version of Proposition~\ref{prop:Wiener} to the algebra~$Y$ from Definition~\ref{def:VB}. 
We noted just below Lemma~\ref{lem:Ynorms} that $Y$ does not have the structure of and $L^{1}_{y}$ convolution algebra taking values in the bounded operators on some Banach space. This prevents us from simply citing the abstract Wiener theorems from~\cite{bec,becgol}.  Assuming that $0$ energy is regular,  we have equation~\eqref{2.37}, viz.
\be\nn
(I + T_{1+}^\eps) \oast (I - T_+^\eps) = (I - T_+^\eps) \oast (I + T_{1+}^\eps) = I
\ee
This holds in the algebra $Z$, uniformly in $\eps\ge0$, see Lemma~\ref{lm2.4}.  This guarantees that $(I + T_{1+}^\eps)^{-1}=I - T_+^\eps$ in~$Z$. We now wish to show that this relation also holds in~$Y$ and this requires a Wiener theorem. For this it is natural that we begin by taking the Fourier transform relative to the variable~$y$ of~\eqref{2.37}. $T_{1+}^{\eps}$ with $\eps=0$ refers to the limit $\eps\to0+$. 

\begin{lemma}
\lb{lem:FTinvert}
Let $V\in  B^{\frac12}$ and assume that $0$ energy is regular of $H=-\Delta+V$. Then for any $\eta\in \R^3$, the operator $I+\what{T^\eps_{1+}}(\eta)$ 
is invertible in $\B(L^\infty)$ (the bounded operators on $L^\infty$)
and 
\EQ{ \label{eq:Inv0}
\sup_{\eps\ge0}\sup_{\eta\in\R^3} \big\| (I+\what{T^\eps_{1+}}(\eta))^{-1}\big\|_{\B(L^\infty)} <\infty
}
Moreover,  in $\B(L^\infty)$ one has the identity 
\EQ{\label{eq:Inv1}
(I+\what{T^\eps_{1+}}(\eta))^{-1} = I - \what{T^\eps_+}(\eta) \qquad\forall\; \eta\in \R^3
}
The second term on the right-hand side satisfies  
\EQ{ \label{eq:Inv2}
\sup_{\eps\ge0}\sup_{\eta\in\R^3} \big\| \what{T^\eps_{+}}(\eta) \big\|_{\B(V^{-1}B,L^\infty)} <\infty
}
where $V^{-1}B$ is defined with  any $\sigma\ge\frac12$.   In fact, with $M_{0}$ as in \eqref{eq:Res Bd}, 
\EQ{ \label{eq:Inv2*}
\sup_{\eps\ge0}\sup_{\eta\in\R^3} \big\| \what{T^\eps_{+}}(\eta) \big\|_{\mc FY}  \les \|V\|_{B^{\sigma}}  M_{0}
}
If $V\in B^{\sigma}$, $\frac12<\sigma<1$, then uniformly in $\eps>0$, the map $\eta\mapsto \what{T_{1+}^{\eps}}(\eta)$ is uniformly H\"older continuous as a map $\R^{3}\to\B(V^{-1}B,L^{\infty})$,
and therefore also as a map $\R^{3}\to\mc FY=\B(L^{\infty})\cap \B(V^{-1}B)$. Quantitatively speaking, one has 
\EQ{
\label{eq:Hoelder}
\| \what{T_{1+}^{\eps}}(\eta) - \what{T_{1+}^{\eps}}(\tilde\eta) \|_{\mc FY}\les |\eta-\tilde\eta|^{\rho} \|V\|_{B^{\sigma}}
}
where $0<\rho=\sigma-\frac12$. 
\end{lemma}
\begin{proof}
By equation \eqref{2.29}
\[
R_V(z)=(I+R_0(z)V)^{-1}R_0(z) \qquad\forall\; \Im z>0
\]
By Lemma~\ref{lemma2.3},
\[
\sup_{\eps\ge0}\sup_{\eta\in\R^3} \big\| (I+R_0(|\eta|^2\pm i\eps )V  )^{-1}\big\|_{\B(L^\infty)} <\infty
\]
Since $R_0(|\eta|^2\pm i\eps ): L^{\frac32,1}\to L^\infty$ uniformly in $\eta,\eps$ and $\dot B^{\frac12}\embed L^{\frac32,1}$ this implies that 
\[
\sup_{\eps\ge0}\sup_{\eta\in\R^3} \big\| R_V(|\eta|^2\pm i\eps ) V\big\|_{\B(V^{-1}B,L^\infty)} <\infty
\]
Let $(M_\eta f)(x) = e^{-i\eta\cdot x}f(x)$. Then by Lemma~\ref{lem:FTe}, 
\EQ{\label{eq:TTT1}
\what{T^\eps_{1+}}(\eta) = M_\eta^{-1} R_0(|\eta|^2-i\eps)V M_\eta,\qquad \what{T^\eps_{+}}(\eta) = M_\eta^{-1} R_V(|\eta|^2-i\eps)V M_\eta
}
Passing to the Fourier transform of \eqref{2.37} therefore yields 
\EQ{\nn
(I + \what{T^\eps_{1+}}(\eta) )\circ (I- \what{T^\eps_{+}}(\eta)) &= I\\ 
 (I + \what{T^\eps_{1+}}(\eta) )^{-1} &= I - M_\eta^{-1} R_V(|\eta|^2-i\eps)V M_\eta
 }
as equations in $\B(L^\infty)$, and so \eqref{eq:Inv1}, \eqref{eq:Inv2} hold. 

For the uniform continuity compute
\EQ{\label{eq:R0diff}
&|[R_0(|\eta|^2-i\eps)V-R_0(|\tilde\eta|^2-i\eps)V](x_0,x_1) | \\
&\les  \frac{\min(1,|\eta-\tilde\eta||x_0-x_1|)}{|x_0-x_1|}|V(x_0)|  \\
&\les |\eta-\tilde\eta|^\rho |x_0-x_1|^{-1+\rho} |V(x_0)| 
}
where we take $\rho=\sigma-\frac12\in (0,1)$. By Lemma~\ref{lem:Balpha interpol},  $B^\sigma\embed L^{\frac{3}{2+\rho},1}$,  and $|x|^{-1+\rho}\in L^{\frac{3}{1-\rho},\infty}=(L^{\frac{3}{2+\rho},1})^{*}$, we conclude by means of~\eqref{eq:Ho2} that 
\[
\| R_0(|\eta|^2-i\eps)V-R_0(|\tilde\eta|^2-i\eps)V \|_{\B(V^{-1}B, L^\infty)} \les  |\eta-\tilde\eta|^\rho
\]
The second line in \eqref{eq:R0diff} follows from, with $a>0$, and uniformly in $\eps>0$, 
\EQ{
& \big| e^{-ia\sqrt{|\eta|^2-i\eps}} - e^{-ia\sqrt{|\tilde \eta|^2-i\eps}} \big | \\
&\le \min(2,  a |\sqrt{|\eta|^2-i\eps} - \sqrt{|\tilde \eta|^2-i\eps}|) \\
& \le 2  \min\Big( 1, a\frac{||\eta|^2-|\tilde\eta|^2|}{ |\sqrt{|\eta|^2-i\eps} + 
\sqrt{|\tilde\eta|^2-i\eps}|}\Big) \le 2 \min(1, a|\eta-\tilde \eta|)
}
Here $\Im \sqrt{|\eta|^2-i\eps}<0$, $\Im \sqrt{|\tilde\eta|^2-i\eps}<0$.  

In view of~\eqref{eq:TTT1}, we next need to bound the differences involving the terms $M_{\eta}$ as $\eta$ changes. Thus, 
\EQ{ 
|(M_{\eta}f-M_{\tilde\eta}f)(x)| & \le \min(2,|\eta-\tilde\eta||x|)|f(x)| \\
&\le 2 |\eta-\tilde\eta|^{\rho}|x|^{\rho} \, |f(x)|
}
We absorb the $|x_{0}|^{\rho}$ factor  into $|V(x_{0})|$. 
For the exterior operator $M_{\eta}^{-1}$ acting in the variable~$x_{1}$, we write $|x_{1}|^{\rho}\les |x_{0}|^{\rho} + |x_{1}-x_{0}|^{\rho}$. 
The first term is passed onto $V$, whereas the second is absorbed as in~\eqref{eq:R0diff}. 

To summarize, 
\[
\| \what{T^\eps_{1+}}(\eta) f - \what{T^\eps_{1+}}(\tilde \eta) f\|_{L^{\infty}} \le C|\eta-\tilde\eta|^{\rho} \|f\|_{V^{-1}B}
\]
which establishes uniform continuity. 
\end{proof}

We now state the Wiener theorem  in  the  algebra~$Y$.   The conditions \eqref{eq:Sass1}, \eqref{eq:Sass2} in the following proposition are precisely 
the two properties of $L^{1}$ functions that made the proof of the scalar Wiener's theorem above work (here $S^{N}$ denotes $N$-fold composition of $S$ with itself using~$\oast$, and $\chi$ is the function from above).  Indeed, \eqref{eq:L1klein} corresponds to~\eqref{eq:Sass1} with $\eps=R^{-1}$, and \eqref{eq:Sass2} will allow us to obtain an analogue of~\eqref{eq:claim}. This is natural, as~\eqref{eq:Sass2} localizes in $y$ and
therefore regularizes in~$\eta$ which makes the essential discretization property in~$\eta$ possible, cf.~\eqref{eq:sumg1}.  Throughout, the standard convolution symbol $\ast$ means convolution relative to the $y$-variable.  Finally, the pointwise invertibility condition of the Fourier transform is modeled after Lemma~\ref{lem:FTinvert}.

\begin{prop}
\lb{prop:YWiener}
Let $V\in B^{\sigma}$ where $\frac12\le \sigma<1$, and define the algebra $Y$ with this value of $\sigma$. 
Suppose that $S\in Y$ satisfies, for some $N\ge1$
\begin{align}
\label{eq:Sass1}
\lim_{\eps\to0} \|\eps^{-3}{\chi}(\cdot/\eps)\ast S^{N} - S^{N}\|_{Y} &= 0  \\
\lim_{L\to\infty} \|(1-\hat{\chi}(y/L))S(y)\|_{Y} &=0\label{eq:Sass2}
\end{align}
Assume that  $I+\hat{S}(\eta)$ has an inverse in $\B(L^\infty)$ of the form $(I+\hat{S}(\eta))^{-1}= I + U(\eta)$, with  $U(\eta)\in \mc FY$ for all $\eta\in\R^{3}$, and uniformly so, i.e., 
\EQ{
\lb{eq:Ubd}
\sup_{\eta\in\R^3} \|U(\eta)\|_{\mc FY}<\infty 
}
Finally, assume that $\eta\mapsto \hat{S}(\eta)$ is uniformly continuous as a map $\R^{3}\to \B(L^{\infty})$. 
 Then it follows that $I+S$ is invertible in~$Y$ under~$\oast$. 
\end{prop}
\begin{proof}
We need to construct $\calL\in Y$ with the property that 
\EQ{\label{eq:SL}
(I+\what{\calL}(\eta))\circ (I+\hat{S}(\eta))=I\qquad \forall\; \eta\in\R^{3}
}
For $|\eta|\ge 2R$ this is the same as 
\EQ{\nn
(I+\what{\calL}(\eta)) \circ (I+\hat{S}(\eta)-\hat{\chi}(\eps\eta)\hat{S}(\eta))=I
}
with $\eps=R^{-1}$.  Taking $R$ so large that 
\EQ{\label{eq:Rgross}
\| \eps^{-3}{\chi}(\cdot/\eps)\ast S^{N} - S^{N}\|_{Y}< 2^{-N}
}
by \eqref{eq:Sass1}, we can write with $\mu_{\eps}:=-\delta_{0}+\eps^{-3}{\chi}(\cdot/\eps)$, 
\EQ{
(I - \mu_{\eps} \ast S)^{-1} &= (I - \mu_{\eps}^{N} \ast S^{N})^{-1} \oast \big(I+\sum_{\ell=1}^{N-1} \mu_{\eps}^{\ell} \ast S^{\ell}\big) \\
&= \Big(I+\sum_{n=1}^{\infty} \mu_{\eps}^{nN} \ast S^{nN}\Big) \oast \big(I+\sum_{\ell=1}^{N-1} \mu_{\eps}^{\ell} \ast S^{\ell}\big)
}
and the infinite series converges in $Y$ since  $\|\mu_{\eps}\|_{Y}\le 2$ and 
\[
\|\mu_{\eps}^{N} \ast S^{N}\|_{Y} \le \|\mu_{\eps}\|_{Y}^{N-1} \| \mu_{\eps}\ast S^{N}\|_{Y} \le 2^{N-1} 2^{-N} =\frac12
\]
Then 
\EQ{\label{eq:Laussen}
\calL:= \Big(I+\sum_{n=1}^{\infty} \mu_{\eps}^{nN} \ast S^{nN}\Big) \oast \big(I+\sum_{\ell=1}^{N-1} \mu_{\eps}^{\ell} \ast S^{\ell}\big)
-I \in Y
}
has the property that \eqref{eq:SL} holds for $|\eta|\ge 2R$.   Note that the scalar convolution does act commutatively relative to~$\oast$, which itself is not commutative (due to the non-commutativity of operator composition). 

\smallskip

Using the same patching method as in the scalar Wiener theorem above, it suffices to construct a local solution of \eqref{eq:SL} on  $|\eta|\le 3R$.
As in the  proof of the scalar Wiener theorem,  for any $1>\eps>0$,  $\eta_{0}\in\R^{3}$, 
\EQ{\nn 
\omega_{\eps,\eta_{0}}(x) &= e^{iy\cdot\eta_{0}} \eps^{3} \chi(\eps y),\quad 
\what{\omega_{\eps,\eta_{0}}}(\eta) = \what{\chi}(\eps^{-1}(\eta-\eta_{0})) 
}
We claim that
\EQ{
\label{eq:claimS} 
\sup_{\eta_{0}\in\R^{3}} \big\| S \ast \omega_{\eps,\eta_{0}} - \hat{S}(\eta_{0}) \omega_{\eps,\eta_{0}}\big \|_{Y} \to 0\text{\ \ as\ \ }\eps\to0
}
By Lemma~\ref{lem:6prop}, property ii), 
\[
\|S \ast \omega_{\eps,\eta_{0}}\|_{Y}\le \|S\|_{Y}\|\omega_{\eps,\eta_{0}}\|_{1} \les \|S\|_{Y}
\]
uniformly in $\eps>0$, $\eta_{0}\in\R^{3}$. 
By properties iv) and v) of the same lemma, 
\EQ{\nn
 \|\hat{S}(\eta_{0}) \omega_{\eps,\eta_{0}}\big \|_{Y} &\les \|\hat{S}(\eta_{0})\|_{\mc FY}  \|\omega_{\eps,\eta_{0}}\|_{1}  \les \|S\|_{Y}
}
To prove the claim,  we may therefore assume that $S(y)=0$ if $|y|\ge L$ for some $L$, using \eqref{eq:Sass2}. 
With this in mind, we compute 
\EQ{\lb{eq:calDdef}
\calD &:= (S\ast \omega_{\eps,\eta_{0}}) (y)(x_{0},x_{1}) - \hat{S}(\eta_{0})(x_{0},x_{1}) \omega_{\eps,\eta_{0}} (y) \\
& = \int_{\R^{3}} S(u)(x_{0},x_{1}) e^{i(y-u)\cdot\eta_{0}} \eps^{3} [\chi(\eps(y-u)) - \chi(\eps y) ]\, du
}
We begin estimating the $\|\calD\|_{Z}$ term in the $Y$-norm. Thus, 
\EQ{\label{eq:DZ}
\| \calD\|_{Z} &= \sup_{\eta} \| \mc F_{y} \calD(x_{0},x_{1},\eta)\|_{L^{\infty}_{x_{1}}L^{1}_{x_{0}}} \\
& = \sup_{\eta} \big\| (\hat{S}(\eta)-\hat{S}(\eta_{0}))(x_{0},x_{1}) \what{\omega_{\eps,\eta_{0}}}(\eta) \big\|_{\B(L^{\infty)}}
}
By assumption of uniform continuity of $\hat{S}(\eta)$ as a map from $\R^{3}$ to $\B(L^{\infty})$ this tends to $0$ uniformly in $\eta_{0}$ as $\eps\to0$. 

 Next,  assume that $\|f\|_{V^{-1}B}\le 1$. With $f\calD$ referring  to the contraction of $\calD$ by $f$ in the $x_{0}$-variable,
 \EQ{\lb{eq:832}
\| f\calD\|_{X_{x_{1},y}}  &\les \int_{\R^{3}} \|(fS)(u)(x_{1})\|_{V^{-1}B_{x_{1}}}   \eps^{3} |\chi(\eps(y-u)) - \chi(\eps y) |\,dy du \\
&\les \eps L  \sup_{x_{1}} \int_{\R^{3}}\|(fS)(u)(x_{1})\|_{V^{-1}B_{x_{1}}}  \, du \les \eps L \| fS\|_{X_{x_{1},y}}\\
&\les \eps L \|S\|_{\B(V^{-1}B_{x_{0}}, X_{x_{1},y} )} \les \eps L \|S\|_{Y}. 
}
This concludes the proof of our claim~\eqref{eq:claimS}. 

Fix some $\eta_{0}\in\R^{3}$, and define $H_{\eta_{0}}\in Y$ via the relation
\EQ{
\lb{eq:Heta0}
[I+\calD+U(\eta_{0})\circ \calD]^{-1} = I+H_{\eta_{0}}
}
where $U(\eta_{0})$ is as in the statement of the proposition. Note that by Lemma~\ref{lem:6prop}, property~vi), the composition $U(\eta_{0})\circ \calD\in Y$. 
By \eqref{eq:Ubd} and~\eqref{eq:claimS} there exists $\eps>0$ so that 
irrespective of the choice of~$\eta_{0}$ we have 
\EQ{\label{eq:U0calD}
\|\calD+U(\eta_{0})\circ \calD\|_{Y}<\frac12
}
which guarantees that $H_{\eta_{0}}\in Y$ is well-defined from~\eqref{eq:Heta0} via a Neumann series. Moreover, $\|  H_{\eta_{0}}\|_{Y}\le 1$.  The significance of $H_{\eta_{0}}$ lies with the following property: if $\what{\omega_{\eps,\eta_{0}}}(\eta)=1$, then 
\EQ{
\label{eq:bend}
(I+\hat{S}(\eta)) \circ (I+\widehat{H_{\eta_{0}}}(\eta))=I+\hat{S}(\eta_{0}) \text{\ \ in \ \ } \B(L^{\infty})
}
In fact, \eqref{eq:Heta0} is equivalent with the equation in $Y$
\EQ{
 [I+\calD+U(\eta_{0})\circ \calD]\oast (I+H_{\eta_{0}} ) =I
}
Taking the Fourier transform in $y$ yields
\EQ{\lb{eq:UH1}
 (I+\hat{\calD} + U(\eta_{0})\circ \hat{\calD}) \circ (I+\what{H_{\eta_{0}}}) = I  \text{\ \ in \ \ } \B(L^{\infty})
}
By assumption that $(I+\hat{S}(\eta_{0}))^{-1}=I+U(\eta_{0})$ in $\B(L^{\infty})$, \eqref{eq:UH1} is the same as 
\EQ{\lb{eq:UH2}
 (I+\hat{S}(\eta_{0}) + \hat{\calD}(\eta)) \circ (I+ \what{H_{\eta_{0}}})= I+\hat{S}(\eta_{0})  \text{\ \ in\ \ }\B(L^{\infty})
}
If $\what{\omega_{\eps,\eta_{0}}}(\eta)=1$, then \eqref{eq:calDdef} implies that \eqref{eq:UH2} is the same as
\[
 (I+\hat{S}(\eta)) \circ (I+ \what{H_{\eta_{0}}}(\eta)) = I+\hat{S}(\eta_{0})
\]
which is \eqref{eq:bend}. 
Next, define 
\EQ{\lb{eq:cLeta0}
\calL_{\eta_{0}} := U(\eta_{0})\omega_{\eps,\eta_{0}} + H_{\eta_{0}} +  H_{\eta_{0}}  \circ U(\eta_{0}) \in Y
}
Then $\what{\omega_{\eps,\eta_{0}}}(\eta)=1$ implies that 
\EQ{\nn
\what{\calL_{\eta_{0}}} (\eta) &= U(\eta_{0})+\what{H_{\eta_{0}}}(\eta) +  \what{H_{\eta_{0}}}(\eta) \circ U(\eta_{0}) \\
&=  (I+\what{H_{\eta_{0}}}(\eta) )\circ (I+U(\eta_{0})) - I  \\
&= (I+\what{H_{\eta_{0}}}(\eta) ) \circ (I+\hat{S}(\eta_{0}))^{-1}  - I
}
which further gives
\[
(I+\what{\calL_{\eta_{0}}}(\eta)) \circ (I+\hat{S}(\eta_{0}))= I+\what{H_{\eta_{0}}}(\eta)
\]
for all $\eta$ near $\eta_{0}$. Using~\eqref{eq:bend}  we infer that 
\[
(I+\what{\calL_{\eta_{0}}}) \circ (I+\hat{S})= I
\]
which solves \eqref{eq:SL} near $\eta_{0}$. Note that the size of the neighborhood is uniform in $\eta_{0}$. 
Using a partition of unity as we did towards the end of the proof of the scalar Wiener theorem, we can patch up these local solutions to a solution on the ball $|\eta|\le 3R$.    

Thus, we have constructed a left inverse, i.e., $\calL\in Y$ with 
$$
(I+\calL)\oast (I+S) = I
$$
In the same fashion we can construct a right inverse, $\wtil \calL\in Y$ with $$(I+S)\oast (I+\wtil \calL) =I$$
But then 
$$
I+\calL=(I+\calL)\oast (I+S) \oast (I+\wtil \calL) = I+\wtil \calL
$$ 
whence $\calL=\wtil\calL$. So $I+S$ is invertible  in $Y$, as claimed. 
\end{proof}

The proof of this Wiener theorem implies the following quantitative version, by means of which we can control the norm of the inverse. 

\begin{cor}
\lb{cor:YWiener}
Under the same hypotheses  as in Proposition~\ref{prop:YWiener},  we let 
\EQ{
\lb{eq:Ubd*}
M_{1}:=1+\sup_{\eta\in\R^3} \|U(\eta)\|_{\mc FY}  
}
Let $0<\eps_{0}<1$, and $L_{0}>1$ satisfy
\begin{align}
\| \eps_{0}^{-3}{\chi}(\cdot/\eps_{0})\ast S^{N} - S^{N}\|_{Y} &< 2^{-N} \label{eq:Bed1} \\
 \|(1-\hat{\chi}(y/L_{0}))S(y)\|_{Y} &  \le cM_{1}^{-1}   \label{eq:Bed2} 
\end{align}
where $c$ is a suitable small absolute constant. Furthermore, let $1>\eps_{1}>0$  be such that 
\EQ{\lb{eq:Bed3}
 \sup_{|\eta-\eta_{0}|\le \eps_{1}}\big\| \hat{S}(\eta)-\hat{S}(\eta_{0})  \big\|_{\B(L^{\infty)}} & \le cM_{1}^{-1}
}
 Then the inverse of $I+S$  in~$Y$ exists and satisfies
 \EQ{\label{eq:Sinvnorm}
 \| (I+S)^{-1}\|_{Y}\les \eps_{0}^{-3}\big( \eps_{1}^{-3}+ L_{0}^{3}M_{1}^{3}\|S\|_{Y}^{3}\big) M_{1}
 }
\end{cor}
\begin{proof}
Condition \eqref{eq:Rgross} is precisely \eqref{eq:Bed1} with $\eps=\eps_{0}$. Then on $|\eta|\ge R=\eps_{0}^{-1}$ we obtain a solution $\calL$ of norm $\|\calL\|_{Y}\les 1$, see~\eqref{eq:Laussen}.  In view of~\eqref{eq:DZ} we have $\| \calD\|_{Z}\ll M_{1}^{-1}$ by~\eqref{eq:Bed3}. Setting $\eps=\eps_{2}:= c(M_{1}L_{0}\|S\|_{Y})^{-1}$ 
in~\eqref{eq:832} with $L=L_{0}$ we obtain 
\[
\|\calD\|_{\B(V^{-1}B_{x_{0}}, X_{x_{1},y} )} \les \eps_{2}L_{0}\|S\|_{Y}\ll M_{1}^{-1}
\]
if we choose $c$ small.   Hence, $\eps=\min(\eps_{1},\eps_{2})=:\eps_{3}$ guarantees that $\|\calD\|_{Y}\ll M_{1}$, uniformly in $\eta_{0}\in\R^{3}$. 
This further insures that \eqref{eq:U0calD} holds, which defines $H_{\eta_{0}}$ via \eqref{eq:Heta0} with $\|H_{\eta_{0}}\|_{Y}\le 1$.  
The local inverse $\calL_{\eta_{0}}$ given by~\eqref{eq:cLeta0} satisfies $\| \calL_{\eta_{0}}\|_{Y}\les M_{1}$. 

Patching together these local solutions requires $\les R^{3}\eps_{3}^{-3}$ choices of $\eta_{0}$ over which we sum up this $M_{1}$ bound on $\calL$. 
Together with the solution exterior to the $R$-ball, 
the
cumulative  bound on the global inverse amounts to
\EQ{
\|\calL\| &\les 1 + R^{3}\eps_{3}^{-3} M_{1} \les \eps_{0}^{-3} (\eps_{1}^{-3}+\eps_{2}^{-3})M_{1} \\
&\les \eps_{0}^{-3}\big( \eps_{1}^{-3}+ L_{0}^{3}M_{1}^{3}\|S\|_{Y}^{3}\big) M_{1}, 
} 
as claimed. 
\end{proof}


\section{The proof of Theorem~\ref{thm:struct} and~\ref{thm:structQ}}
\label{sec:proof1}

We first verify the conditions \eqref{eq:Sass1}, \eqref{eq:Sass2} for $T_{1+}^\eps$. It is precisely at this point that we need to define the algebra $Y$ using $V^{-1}B$
with $\sigma>\frac12$. 

\begin{lemma}
\lb{lem:T1cond}
Let $V\in B^{2\sigma}$ where $\frac12<\sigma<1$ is arbitrary but fixed. Then $S=T_{1+}^{\eps} \in Y$ satisfies, for sufficiently large $N\ge1$
\begin{align}
\label{eq:Sass1T}
\lim_{\delta\to0} \|\delta^{-3}{\chi}(\cdot/\delta)\ast S^{N} - S^{N}\|_{Y} &= 0  \\
\lim_{L\to\infty} \|(1-\hat{\chi}(y/L))S(y)\|_{Y} &=0\label{eq:Sass2T}
\end{align}
uniformly in $\eps\ge0$. 
\end{lemma}
\begin{proof}
We begin with \eqref{eq:Sass2T}. By definition of the $Z$ space
\EQ{\lb{eq:LSZ}
\|(1-\hat{\chi}(y/L))S(y)\|_{Z} &=  \sup_{\eta\in\R^3} \big\| [(\delta_0-L^3\chi(\cdot L))\ast \what{T_{1+}^{\eps}}](\eta) \big\|_{\B(L^\infty)} 
}
By Lemma~\ref{lem:FTinvert}, $\what{T_{1+}^{\eps}}(\cdot)$  is uniformly continuous as a map $\R\to \B(L^\infty)$, which precisely 
means that~\eqref{eq:LSZ} vanishes in the limit $L\to\infty$. 

\noindent Next, \eqref{eq:Kest2weighted} and with $\sigma=\frac12+\gamma>\frac12$, 
\begin{align}\label{eq:gamma schranke}
\int_{\R^{3}} \la y\ra^{\gamma}\| V(x) K_{1+}^\eps (x,y)\|_{B^{\sigma}_x} \, dy & \les \| V\|_{B^{1+2\gamma}} \| V\|_{B^{\sigma}}     
\end{align}
Consequently,  since $1+2\gamma = 2\sigma$, 
\EQ{\nn 
\sup_{\eps>0} \int_{\R^{3 }} \one_{[|y|\ge L]} \big\|  V(x) (f T_{1+}^{\eps})(x,y)\, \big\|_{B^{\sigma}} \, dy\les  L^{-\gamma} \|fV\|_{B^{\sigma}} \|V\|_{B^{2\sigma}}
}
which is equivalent with 
\EQ{\nn 
\sup_{\eps>0}\Big\|  \one_{[|y|\ge L]}(f T_{1+}^{\eps})(x,y)\Big\|_{X_{x,y}} \les  L^{-\gamma} \|f\|_{V^{-1}B} \|V\|_{B^{2\sigma}}
}
and thus also with, for all $L\ge1$,  
\EQ{\label{eq:LfT}
\sup_{\eps>0}\Big\|  \one_{[|y|\ge L]} T_{1+}^{\eps}(x,y)\Big\|_{\B(V^{-1}B,X_{x,y})} \les  L^{-\gamma}  \|V\|_{B^{2\sigma}}
}
with an absolute implicit constant.  In conjunction with \eqref{eq:LSZ} this proves~\eqref{eq:Sass2T}. 

To prove \eqref{eq:Sass1T} we may therefore assume that $T(y)=0$ for $|y|\ge L$, with some large~$L$.  Using Lemma~\ref{lem:Ynorms} we estimate
\begin{align}
& \big\| \one_{[|y|\le L]} (\delta^{-3}{\chi}(\cdot/\delta)\ast S^{N} - S^{N} ) \big \|_Y \nn   \\
&\les \big\| \one_{[|y|\le L]} (\delta^{-3}{\chi}(\cdot/\delta)\ast S^{N} - S^{N}) \big \|_{L^{1}_{y}\B(V^{-1}B_{x_{0}}, L^{\infty}_{x_{1}})}  \nn \\
&\les L^{3} \big \|  \delta^{-3}{\chi}(\cdot/\delta)\ast S^{N} - S^{N} \big \|_{L^{\infty}_{y}\B(V^{-1}B_{x_{0}}, L^{\infty}_{x_{1}})}  \nn  \\
&\les L^{3}  \| (1-\hat{\chi}(\delta \eta)) \mc F_{y}{S}^{N}(\eta)\|_{L^{1}_{\eta}\B(V^{-1}B_{x_{0}}, L^{\infty}_{x_{1}})}  \label{eq:L3}
\end{align}
By \eqref{eq:t1'}, 
\EQ{\nn 
\mc F_y (T_{1+}^\eps)^{N} (x_0, x_1, \eta) = e^{ix_0\eta} \big(R_0(|\eta|^2 - i\eps ) V\big)^{N}(x_0, x_1) e^{-ix_1\eta}.
}
We claim that if $N$ is large enough, then uniformly in $\eps>0$, 
\EQ{\label{eq:dec claim}
\|  \mc F_{y}{S}^{N}(\eta)\|_{\B(V^{-1}B_{x_{0}}, L^{\infty}_{x_{1}})} \le C |\eta|^{-4}
}
If so, then \eqref{eq:L3} vanishes in the limit $\delta\to0$, and \eqref{eq:Sass1T} holds. 
Recall the Stein-Tomas type bound for the free resolvent, see~\cite{krs}, 
\[
\|R_{0}(1+i\eps)f\|_{L^{4}(\R^{3})}\le C\|f\|_{L^{\frac43}(\R^{3})}
\]
uniformly in $\eps>0$. By scaling, this implies 
\EQ{\label{eq:krs}
\sup_{\eps>0}\|R_{0}(\lambda^{2}+i\eps)\|_{L^{\frac43}(\R^{3})\to L^{4}(\R^{3})}\le C\lambda^{-\frac12}
}
for all $\lambda>0$, and therefore 
\EQ{\label{eq:krs4}
\sup_{\eps>0}\|(R_{0}(\lambda^{2}+i\eps)V)^{8}\|_{L^{4}(\R^{3})\to L^{4}(\R^{3})}\le C\lambda^{-4}
}
where we used that $V\in L^{2}(\R^{3})$.  Recall $R_{0}(\lambda^{2}+i\eps)V: V^{-1}B\to L^{\infty}$ uniformly in $\lambda, \eps$, see Lemma~\ref{lm2.4}. Next,  $V\in L^{\frac65}\cap L^{2}\embed L^{\frac43}$, see Lemma~\ref{lem:Balpha interpol}.  Therefore, 
\EQ{\label{eq:krs5}
\sup_{\eps>0}\|(R_{0}(\lambda^{2}+i\eps)V)^{9}\|_{V^{-1}B\to L^{4}(\R^{3})}\le C\lambda^{-4}
}
for all $\lambda\ge1$.  Since $V\in L^{\frac{12}{7}, 1}(\R^{3})$, see \eqref{eq:BdotL}, we have $M_{V}:L^{4}\to L^{\frac65,1}$ by~\eqref{eq:Ho1} where $M_{V}$ is the multiplication operator by~$V$. Furthermore, by \eqref{eq:Y2}, one has $R_{0}: L^{\frac65,1}\to L^{6,1}$ and $M_{V}: L^{6,1}\to L^{\frac32,1}$ since $V\in L^{2}$. Finally, $R_{0}:L^{\frac32,1}\to L^{\infty}$. To summarize, 
\EQ{\label{eq:krs7}
\sup_{\eps>0}\|(R_{0}(\lambda^{2}+i\eps)V)^{11}\|_{V^{-1}B\to L^{\infty}(\R^{3})}\le C\lambda^{-4}
}
and \eqref{eq:dec claim} holds with $N=11$. 
\end{proof}

We are ready to prove the  structure theorem on the wave operators for potentials $V\in B^{2\sigma}$, $\sigma>\frac12$. 

\begin{proof}[Proof of Theorem~\ref{thm:struct}]
All assumptions of our abstract Wiener theorem, Proposition~\ref{prop:YWiener}, hold for $T_{1+}^{\eps}$. And they do so uniformly in $\eps>0$. 
Thus we can invert $I+T_{1+}^{\eps}$ in $Y$. By uniqueness of the inverse in the ambient algebra~$Z$, see~\eqref{2.37},  this inverse is given by 
\EQ{\label{eq:waschno}
(I+T_{1+}^{\eps})^{-1}=I-T_{+}^{\eps}, \qquad M_{2}:=\sup_{\eps>0} \|T_{+}^{\eps}\|_{Y} <\infty
}
Hence
\EQ{\label{eq:T+eps}
T_{+}^{\eps} = I - (I+T_{1+}^{\eps})^{-1} & = (I+T_{1+}^{\eps})^{-1} \oast T_{1+}^{\eps} = (I-T_{+}^{\eps})\oast T_{1+}^{\eps} \\
&=  T_{1+}^{\eps} \oast  (I-T_{+}^{\eps})
}
By Lemma~\ref{lemma2.1}, eq.~\eqref{2.20},  for any Schwartz function $f$, 
\[
(W_{+}^{\eps}f)(x) = f(x)- \int_{\R^{3}} (\one_{\R^{3}}T_{+}^{\eps})(x,y)f(x-y)\, dy,
\]
cf.~\eqref{eq:1Tn}.  Set $\frak X_{+}^{\eps}(x,y):= (\one_{\R^{3}}T_{+}^{\eps})(x,y)$ which satisfies, by~\eqref{eq:waschno} and with $X=L^{1}_{y}V^{-1}B_{x}$, 
\EQ{\label{eq:waschno*}
\sup_{\eps>0}\| \frak X_{+}^{\eps}\|_{X}\le \sup_{\eps>0}\| T_{+}^{\eps}\|_{Y}\| \|\one_{\R^{3}}\|_{V^{-1}B} \les M_{2}\|V\|_{B^{2\sigma}} <\infty
}
By \eqref{eq:frakXV}  we have 
\EQ{\nn
\frak X_{V}^{\eps} (x,y) &= -(\one_{\R^{3}} T_{1+}^{\eps})(x,y) \in X \\
\| \frak X_{V}^{\eps} \|_{L^{1}_{y}V^{-1}B} &=  \| \frak X_{V}^{\eps} \|_X \les \|T_{1+}^{\eps}\|_{Y}\|\one_{\R^{3}}\|_{V^{-1}B} \les \|V\|^{2}_{B^{2\sigma}}
}
 and \eqref{eq:T+eps} implies that 
\EQ{\lb{eq:XVT}
\frak X_{+}^{\eps} &=  - \frak X_{V}^{\eps} - \frak X_{+}^{\eps}T_{1+}^{\eps}  
}
where we used the contraction notation~\eqref{contractie} in the final term.  The first term in~\eqref{eq:XVT} has the desired form  by Corollary~\ref{cor:K1+}. For the second term, in analogy with \eqref{eq:Wneps}, we have 
\EQ{
(\frak X_{+}^{\eps}T_{1+}^{\eps})(x,y)  =  \int_{\R^{3}} \frak X^{\eps}_{f^{\eps}_{y'}V}(x,y-y')\, dy', \qquad f^{\eps}_{y'}(x'):= \frak X_{+}^{\eps}(x',y')
}
This kernel operates on Schwartz  functions by contraction, i.e., 
\[
(\frak X^{\eps}_{f^{\eps}_{y'}} \, \phi)(x') = \int_{\R^{3}}  \frak X_{+}^{\eps}(x',y') \phi(x'-y')\, dy'
\]
Formulas \eqref{eq:g1*}, \eqref{eq:gnmeas'} apply here. Viz., from  Corollary~\ref{cor:K1+},  eq.~\eqref{eq:g1},  there exists $g_{1,f^{\eps}_{y'}}^{\eps}(x,dy,\omega)$  so that for every 
$\phi\in \mc S$ one has
\EQ{\label{eq:g1W}
(\frak X^{\eps}_{f^{\eps}_{y'}V}\; \phi)(x) = \int_{\Sph^2 }\int_{\R^3} g_{1,f^{\eps}_{y'}}^{\eps} (x,dy,\omega) \phi(S_\omega x-y)\,  \sigma(d\omega)
}
where  for fixed $x\in\R^3$, $\omega\in \Sph^2 $ the expression $g_{1,f^{\eps}_{y'}}^{\eps}(x,\cdot,\omega)$ is a measure satisfying 
\EQ{\lb{eq:gnmeasW}
\sup_{\eps>0} \int_{\Sph^2 }  \| g_{1,f^{\eps}_{y'}}^{\eps}(x,dy,\omega)\|_{\mes_{y} L^\infty_x}  \, d\omega &\le C\|f^{\eps}_{y'}V\|_{B^{\frac12}} \le C\|f^{\eps}_{y'}V\|_{B^{\sigma}}\\
&= C  \| \frak X_{+}^{\eps}(x',y') \|_{V^{-1}B_{x'}}    
}
The exact same calculations that we carried out in eq.~\eqref{eq:I}--\eqref{eq:IV} now yield 
\EQ{\lb{eq:II*}
h^\eps(x,dy, \omega) := \int_{\R^{3}} g_{1,f^{\eps}_{y'}}^{\eps} (x-y',d(y-S_{\omega}y'),\omega) \, dy'
}
and 
\EQ{\nn
(\frak X_{+}^{\eps}T_{1+}^{\eps} \phi)(x) &=  \int_{\Sph^{2}}\int_{\R^{3}} h^{\eps} (x,dy,\omega) \phi(S_{\omega}x-y)\, d\omega \\
}
where $h^{\eps}$ satisfies
\EQ{\nn 
  \int_{\Sph^2 }  \| h^{\eps}(x,dy,\omega)\|_{\mes_{y} L^\infty_x}  \, d\omega  & \le    C  \int_{\R^3}  \|  \frak X_{+}^{\eps}(x',y') \|_{V^{-1}B_{x'}} \,  dy' \\
& =  C   \|\frak X_{+}^{\eps} \|_{X} \le C\, M_{2}\|V\|_{B^{2\sigma}}
}
uniformly in $\eps>0$. Setting $g^{\eps}:=g_{1}^{\eps} + h^{\eps}$, where $g_{1}^{\eps} $ is from~\eqref{eq:g1}, we arrive at the 
structure formula
 \begin{align}
(W_{+}^{\eps}f)(x) &= f(x)+ \int_{\R^{3}} g^{\eps}(x,dy,\omega)f(S_{\omega}x-y) \nn  \\
\sup_{\eps>0}\int_{\Sph^2 }  \|  g^{\eps}(x,dy,\omega)\|_{\mes_{y} L^\infty_x}  \, d\omega & \les (1+M_{2})\|V\|_{B^{2\sigma}} \label{eq:grand}
\end{align}
Finally, we pass to the limit $\eps\to0$ by Lemma~\ref{lem:Weps}. 

The claim concerning the general Banach space~$X$ (with a different meaning than then one appearing in $Y$) and~\eqref{eq:HW}  in Theorem~\ref{thm:struct} follows from Corollary~\ref{cor:K1+}, eq.~\eqref{eq:g1def}, and~\eqref{eq:II*}. Indeed, \eqref{eq:g1def} shows that $g_1$ has this half-space structure, and the global structure functions is an average of such operators.  To be specific, let $H(\omega,y):= \{ x\in\R^3\::\: (y+2x)\cdot\omega>0\}$. Then by~\eqref{eq:g1def}
\EQ{\lb{eq:1234}
&\Big\| \int_{\Sph^2 } \int_{\R^{3}} g_1(x,dy,\omega)f(S_{\omega}x-y)\, d\omega \Big\|_X \\ &\le  \int_{\Sph^2 } \int_{\R^{3}} \big\| \one_{H(\omega,y)} f(S_\omega x-y)\big\|_X |L_V(y\cdot\omega,\omega)|\, \calH^1_{\ell_\omega}(dy) \\
&\le A \|f\|_X \int_{\Sph^2 } \int_{\R}   |L_V(r,\omega)|\, drd\omega \les A \|V\|_{\dot B^\frac12} \|f\|_X \le AC(V) \|f\|_X.
}
Passing to the limit $\eps\to0$ in $h^\eps$  we have
\[
h(x,dy,\omega) = \int_{\R^3} \one_{H(\omega,y-S_\omega y')}(x-y') L_{f_{y'}V} ((y-S_\omega y')\cdot\omega,\omega) \calH^1_{\ell_\omega}( d (y-S_\omega y') ) \, dy' 
\]
whence in analogy with~\eqref{eq:1234}
\EQ{\nn 
&\Big\| \int_{\Sph^2 } \int_{\R^{3}} h(x,dy,\omega)f(S_{\omega}x-y)\, d\omega \Big\|_X \\ 
&\le A \|f\|_X   \int_{\R^3} \int_{\Sph^2 } \int_{\R}   |L_{f_{y'}V}(r,\omega)|\, drd\omega \, dy' \\
&\les A\|f\|_X  \int_{\R^{3}}\| \frak X_{+}^{\eps}(x',y') \|_{V^{-1}B_{x'}} \, dy'  \les M_{2}\|V\|_{B^{2\sigma}} A\|f\|_X
}
   The claim concerning the variable $x_\omega$, and the associated bound~\eqref{reg}, follow in the same way. 
\end{proof}

In order to obtain quantitative estimates on the structure function $g(x,y,\omega)$, we verify the conditions of Corollary~\ref{cor:YWiener}. 

\begin{lemma}
\lb{lem:T1condQ}
Let $V\in B^{2\sigma}$ where $\frac12+\gamma=\sigma$, with $0<\gamma< \frac{1}{2}$.  Then the conditions in Corollary~\ref{cor:YWiener} hold with
\EQ{\label{eq:Qcond}
 K &:= 1 + \|V\|_{B^{2\sigma}} \\
  L_{0} &=  c^{-1} (KM_{1})^{\frac{1}{\gamma}} \\
 M_{1} &= 1 + K  M_{0} \\
 \eps_{0} &=  cK^{-10-\frac{33}{\gamma}} \\
 \eps_{1} &=  c K^{-2}M_{1}^{-2} 
}
where $M_{0}$ is as in~\eqref{eq:Res Bd}, and $c$ is a small absolute constant. 
Thus, there is a bound
\EQ{\label{eq:M37}
\sup_{\eps>0}\| T_{+}^{\eps}\|_{Y} \les  M_{2}:=K^{37+\frac{105}{\gamma}} (1+M_{0})^{4+\frac{3}{\gamma}} 
}
Therefore, in combination with~\eqref{eq:grand} we obtain 
 \begin{align}
\sup_{\eps>0}\int_{\Sph^2 }  \|   g^{\eps}(x,dy,\omega)\|_{\mes_{y} L^\infty_x}  \, d\omega & \les K^{38+\frac{105}{\gamma}} (1+M_{0})^{4+\frac{3}{\gamma}}  \label{eq:grand*}
\end{align}
\end{lemma}
\begin{proof}
The choice of $M_{1}$ is dictated by \eqref{eq:Inv2*} and \eqref{eq:Ubd*}.  The H\"older bound~\eqref{eq:Hoelder} holds with $\rho=\frac12$:
\EQ{
\label{eq:Hoelder*}
\| \what{T_{1+}^{\eps}}(\eta) - \what{T_{1+}^{\eps}}(\tilde\eta) \|_{\mc FY}\les |\eta-\tilde\eta|^{\frac12} \|V\|_{B^{\sigma}} \ll M_{1}^{-1}
} 
by our choice of $\eps_{1}$. In view of \eqref{eq:LSZ} we need to take $L_{0}\ge \eps_{1}^{-1}$ for the $Z$-part of the $Y$-norm.  For the other part we
use~\eqref{eq:gamma schranke} and~\eqref{eq:LfT} to conclude that in total
\[
L_{0} =\eps^{-1}_{1}+ (M_{1}K)^{\frac{1}{\gamma}} 
\]
suffices. But for $\gamma\le\frac12$ this gives the choice above.  For \eqref{eq:Bed1}, we bound
\EQ{\lb{eq:StildeS}
& \| \eps_{0}^{-3}{\chi}(\cdot/\eps_{0})\ast S^{N} - S^{N}\|_{Y} \\
&\le \| \eps_{0}^{-3}{\chi}(\cdot/\eps_{0})\ast \tilde S^{N} - \tilde S^{N}\|_{Y} + 2\| S^{N} -\tilde S^{N}\|_{Y} \\
&\les \|\tilde S\|^{N-1} \|\eps_{0}^{-3}{\chi}(\cdot/\eps_{0})\ast \tilde S - \tilde S\|_{Y}  + N(\|S\|_{Y}+\|\tilde S\|_{Y})^{N-1}\|S-\tilde S\|_{Y}
}
where $\tilde S = \one_{[|y|\ge L_{1}]}S$, where $L_{1}$ needs to be determined (and the truncation is a smooth one). Note that we saw above that $N=11$ suffices, and that 
\[
 \|S-\tilde S\|_{Y} \les L_{1}^{-\gamma} K
\]
Hence, \eqref{eq:StildeS} yields
\[
\| \eps_{0}^{-3}{\chi}(\cdot/\eps_{0})\ast S^{N} - S^{N}\|_{Y} \les \eps_{0} L_{1}^{3} K^{10} + K^{11} L_{1}^{-\gamma} 
\] 
The first term on the right-hand side follows from \eqref{eq:L3}--\eqref{eq:krs7}.  The optimal choice is $\eps_{0}=KL_{1}^{-\gamma-3}$, and $L_{1}$ must then satisfy
\[
K^{11}L_{1}^{-\gamma}\ll 1 
\] 
which leads to $\eps_{0}$ above. 
\end{proof}

\begin{proof}[Proof of Theorem~\ref{thm:structQ}]
The quantitative estimate~\eqref{eq:gQ} is a restatement of~\eqref{eq:grand*}. 

Let $\tilde V$ be as in the statement of the theorem. The resolvent identity implies that for any $z$ in the upper half-plane
\[
(I + R_0(z) \tilde V)^{-1} = (I + R_0(z)  V)^{-1} - (I + R_0(z) V)^{-1} (\tilde V-V) (I + R_0(z) \tilde V)^{-1}
\]
whence in the operator norm on $L^{\infty}$
\EQ{\nn 
&\| (I + R_0(z) \tilde V)^{-1}\| \\&\le \| (I + R_0(z) V)^{-1}\|  + CM_{0}\|V-\tilde V\|_{L^{\frac32,1}} \| (I + R_0(z) \tilde V)^{-1}\|
}
By \eqref{eq:VVt}
\[
CM_{0}\|V-\tilde V\|_{L^{\frac32,1}} \le  CM_{0}\|V-\tilde V\|_{B^{1+2\gamma}} \le \frac12, 
\]
then 
\[
\sup_{z\not\in \R}\| (I + R_0(z) \tilde V)^{-1}\| \le 2M_{0}
\]
In particular, $H=-\Delta+\tilde V$ satisfies the $0$-energy condition and its structure function $\tilde g$ satisfies a bound comparable to~\eqref{eq:gQ} (by~\eqref{eq:VVt}). 

Note that $g$ and $\tilde g$ are each constructed in their own respective algebras, $Y$ and $\tilde Y$. The former is based on $V^{-1}B$, whereas the latter uses~$\tilde V^{-1}B$. However, if we set $U:= |V|+|\tilde V| \in B^{1+2\gamma}$, then 
\[
\| Vf\|_{B^{1+2\gamma}} + \| \tilde V f\|_{B^{1+2\gamma}} \les \| U f\|_{B^{1+2\gamma}}
\]
Consequently, we can carry out the construction of $g$ and $\tilde g$ simultaneously, using the algebra based on $U$.  We will denote this new algebra by~$Y$. 
Let $T_{1+}^{\eps}$ and $\tilde T_{1+}^{\eps}$ be the operators associated with $V$ and $\tilde V$, respectively. Recall from~\eqref{eq:waschno} that 
\EQ{\nn  
(I+T_{1+}^{\eps})^{-1} & =I-T_{+}^{\eps}, \qquad (I+\tilde T_{1+}^{\eps})^{-1}=I-\tilde T_{+}^{\eps}  
}
whence 
\EQ{
\tilde T_{+}^{\eps} - T_{+}^{\eps} &= (I+T_{1+}^{\eps})^{-1} - (I+\tilde T_{1+}^{\eps})^{-1} \\
&= (I+T_{1+}^{\eps})^{-1} \oast  ( T_{1+}^{\eps} -   \tilde T_{1+}^{\eps}   ) \oast  (I+\tilde T_{1+}^{\eps})^{-1}\\
&= (I-T_{+}^{\eps}) \oast  ( T_{1+}^{\eps} -   \tilde T_{1+}^{\eps}   ) \oast  (I-\tilde T_{+}^{\eps}) \\
\| \tilde T_{+}^{\eps} - T_{+}^{\eps} \|_{Y} &\les M_{2}^{2} \| V-\tilde V\|_{B^{1+2\gamma}} 
}
In the last line we used that same constant  $M_{2}$, up to an absolute factor,  controls $\|T_{+}^{\eps}\|_{Y}$ as well as $\|\tilde T_{+}^{\eps}\|_{Y}$, see~\eqref{eq:M37}.
In view of~\eqref{eq:waschno*}
\EQ{\label{eq:katze}
\sup_{\eps>0} \| {\frak X}_{+}^{\eps} - \tilde{\frak X}_{+}^{\eps} \|_{X} &\le \sup_{\eps>0} \| T_{+}^{\eps}-\tilde T_{+}^{\eps} \|_{Y}  \les M_{2}^{2}
K \|V-\tilde V\|_{B^{2\sigma}} <\infty
}
The calculations \eqref{eq:waschno*}--\eqref{eq:grand} yield 
\[
g^{\eps} - \tilde g^{\eps} = g_{1}^{\eps} -\tilde  g_{1}^{\eps} + h^{\eps} -\tilde h^{\eps}
\]
First, by linearity
\EQ{
\lb{eq:g1g1}
\int_{\Sph^{2}}  \| g_{1}^{\eps} -\tilde  g_{1}^{\eps}\|_{\mes_{y} L^\infty_x}  \, d\omega & \les \| V-\tilde V\|_{B^{2\sigma}}
}
For  $h^{\eps}-\tilde h^{\eps}$ we have, using the notation \eqref{eq:XVT}--\eqref{eq:gnmeasW}, 
\EQ{\label{eq:g shift}
 h^{\eps} -\tilde h^{\eps} &= g_{1,f^{\eps}_{y'}}^{\eps} - \tilde g_{1,\tilde f^{\eps}_{y'}}^{\eps} = g_{1,f^{\eps}_{y'}-\tilde f^{\eps}_{y'}}^{\eps} + (g-\tilde g)^{\eps}_{1,f^{\eps}_{y'}} \\
  f^{\eps}_{y'}(x') & := \frak X_{+}^{\eps}(x',y'),\qquad    \tilde f^{\eps}_{y'}(x')  := \tilde{\frak X}_{+}^{\eps}(x',y') 
}
and $\tilde g_{1}$ refers to the structure functions defined in terms of $\tilde V$ instead of~$V$.  By~\eqref{eq:gnmeasW}, 
\EQ{
& \quad \sup_{\eps>0} \int_{\R^{3}} \int_{\Sph^2 }  \| g_{1,f^{\eps}_{y'}-\tilde f^{\eps}_{y'}}^{\eps}(x,dy,\omega)\|_{\mes_{y} L^\infty_x}  \, d\omega dy' \\ 
&\les \sup_{\eps>0}  
  \| \frak X_{+}^{\eps} - \tilde{\frak X}_{+}^{\eps}\|_{L^{1}_{y}U^{-1}B_{x'}}       \les  \sup_{\eps>0} \|U\|_{B^{\sigma}} \| T_{+}^{\eps} - \tilde T_{+}^{\eps}\|_{Y} \\
  & \les M_{2}^{2}
K^{2} \|V-\tilde V\|_{B^{2\sigma}} 
}
by \eqref{eq:katze}.  The second term in~\eqref{eq:g shift} is bounded by 
\EQ{\label{eq:weh}
& \quad \sup_{\eps>0} \int_{\R^{3}} \int_{\Sph^2 }  \| (g-\tilde g)^{\eps}_{1,f^{\eps}_{y'}} (x,dy,\omega)\|_{\mes_{y} L^\infty_x}  \, d\omega dy'  \\
&\les \sup_{\eps>0} \int_{\R^{3}}   \| (V-\tilde V)(x) \frak{X}_{+}^{\eps}(x,y)\|_{B^{\sigma}_{x}} \,  dy\\
&\les \sup_{\eps>0} \int_{\R^{3}}   \| (V-\tilde V)U^{-1}\one_{[U\ne0]} \|_{\infty }   \| \frak{X}_{+}^{\eps}(x,y) \|_{U^{-1}B_{x}} \, dy  \\
&\les  \| (V-\tilde V)U^{-1}\one_{[U\ne0]} \|_{\infty } M_{2}K
}
see \eqref{eq:waschno*}.  In summary, this concludes the proof of Theorem~\ref{thm:structQ}. 
\end{proof}

For the sake of completeness,  we now  return to the $L^{\infty}_{x}L^{1}_{y}$-norm that we could not include in~$Y$, see~\eqref{eq:T1 ext}. 
From~\eqref{eq:T+eps}, we have 
\EQ{\lb{eq:TTT}
T_{+}^{\eps} = T_{1+}^{\eps} - T_{1+}^{\eps} \oast  T_{+}^{\eps}
}
From \eqref{eq:T1 ext} we have 
\[
\| T_{1+}^{\eps}\|_{\B(U^{-1}B,L^{\infty}L^{1}_{y})}  \les \|V\|_{B^{2\sigma}}
\]
Since 
\[
\| T_{+}^{\eps} \|_{\B(U^{-1}B,L^{1}_{y}U^{-1}B)}  \les M_{2}
\]
we conclude that  (using the contraction notation against a Schwartz function $f$)
\EQ{\nn 
\| f (T_{1+}^{\eps} \oast  T_{+}^{\eps})  \|_{L^{\infty}_{x_{2}}L^{1}_{y}}  &=  \Big\| \int_{\R^{6}}  T_{1+}^{\eps}(x_{1},x_{2},y-y')    (fT_{+})(x_{1},y') \, dx_{1} dy' \Big\|_{L^{\infty}_{x_{2}}L^{1}_{y}} \\
& =  \Big\| \int_{\R^{3}}  ( (fT_{+})(\cdot,y')T_{1+}^{\eps}) (x_{2},y-y')    \,  dy' \Big\|_{L^{\infty}_{x_{2}}L^{1}_{y}}  \\
&\le \int_{\R^{3}}  \big\|   ( (fT_{+})(\cdot,y')T_{1+}^{\eps}) (x_{2},y)  \big\|_{L^{\infty}_{x_{2}}L^{1}_{y}}   \,  dy'   \\
&\le  \int_{\R^{3}} \| T_{1+}^{\eps}\|_{\B(U^{-1}B,L^{\infty}L^{1}_{y})}   \| (fT_{+})(\cdot,y') \|_{U^{-1}B}  \,  dy'    \\
&\le \| T_{1+}^{\eps}\|_{\B(U^{-1}B,L^{\infty}L^{1}_{y})}  \| T_{+}  \|_{\B(U^{-1}B,L^{1}_{y}U^{-1}B)} \|f\|_{U^{-1}B}
}
So it follows that 
\[
\| T_{1+}^{\eps} \oast  T_{+}^{\eps} \|_{\B(U^{-1}B,L^{\infty}L^{1}_{y})}  \les  M_{2}  \|V\|_{B^{2\sigma}}
\]
which in conjunction with \eqref{eq:TTT} implies that 
\EQ{\lb{eq:T+extend}
\| T_{+}^{\eps}\|_{\B(U^{-1}B,L^{\infty}L^{1}_{y})}  \les M_{2}\|V\|_{B^{2\sigma}} 
}
This shows that {\em after the fact}, this norm is also controlled. But it cannot be included in~$Y$ to begin with, since this would render the Wiener theorem above inapplicable (as the asymptotic vanishing in $y$ were then to fail in the $Y$ norm). 
Also note that this does not improve~\eqref{eq:weh}, since factoring  out $\| V-\tilde V\|_{B^{\sigma}}$ would require finiteness of the $L^{1}_{y}L^{\infty}_{x}$-norm, which is not attainable.

\end{document}